\documentclass[12pt]{amsart}  
\usepackage{amssymb,amsmath,amsthm}
\usepackage[english]{babel}
\usepackage{geometry}                
\usepackage{graphicx}
\usepackage{amssymb}
\usepackage{epstopdf}
\usepackage{hyperref}

\title{The coin-turning walk and its scaling limit}

\author{J\'anos Engl\"ander}\thanks{J.E.'s research was supported in part by Simons Foundation Grant  579110.}
\address{Department of Mathematics, University of Colorado at Boulder, Boulder, CO-80309, USA}
\email{janos.englander@colorado.edu}
\urladdr{https://www.colorado.edu/math/janos-englander}

\author{ Stanislav Volkov}\thanks{S.V.'s research was  supported in part by Swedish Research Council grant VR~2014-5157.}
\address{Centre for Mathematical Sciences\\ Lund University\\ Lund 22100-118, Sweden}
\email{s.volkov@maths.lth.se}
\urladdr{http://www.maths.lth.se/~s.volkov/}

\author{Zhenhua Wang}
\address{Department of Mathematics, University of Colorado at Boulder, Boulder, CO-80309, USA}
\email{zhenhua.wang@colorado.edu}
\urladdr{https://www.colorado.edu/math/zhenhua-wang}

\keywords{coin-turning, random walks, scaling limits, time-inhomogeneous Markov-processes, Invariance Principle, cooling dynamics, heating dynamics.}
\subjclass[2010] {60G50, 60F05,60J10}
\date{\today}

\newtheorem*{thma}{Theorem A}

\newtheorem*{propD}{Proposition D1972}

\newtheorem{theorem}{Theorem}
\newtheorem{definition}{Definition}
\newtheorem{lemma}{Lemma}
\newtheorem{proposition}{Proposition}

\newtheorem{assumption}{Assumption}

\newtheorem{remark}{Remark}
\newtheorem{example}{Example}

\newtheorem{problem}{Problem}
\newtheorem{question}{Question}

\newcommand{\Var}{\mathsf{Var}}
\newcommand{\Cov}{\mathsf{Cov}}
\newcommand{\Corr}{\mathsf{Corr}}

\def \a   {\alpha}
\def \b   {\beta}
\def \g   {\gamma}

\def \eps {\varepsilon}
\def\E{{\mathbb{E}}}
\def\N{{\mathbb{N}}}
\def\F{\mathcal{F}}
\def\P{{\mathbb{P}}}
\def\R {{\mathbb{R}}}
\def\Mc{{\mathcal{M}}}
\def\Ic{{\mathcal{I}}}
\def\ds{\displaystyle}
\begin{document}

\begin{abstract}
Let $S$ be the random walk obtained from ``coin turning" with some sequence $\{p_n\}_{n\ge 2}$, as introduced in~\cite{EV2018}. In this paper we investigate the scaling limits of $S$ in the spirit of the classical Donsker invariance principle, both for the heating and for the cooling dynamics.

We prove that an invariance principle, albeit with a non-classical scaling, holds for ``not too small" sequences, the order const$\cdot n^{-1}$ (critical cooling regime) being the threshold. At and below this critical order, the scaling behavior is dramatically different from the one above it. The same order is also the critical one for the Weak Law of Large Numbers to hold.

In the critical cooling regime, an interesting process  emerges: it is a continuous, piecewise linear, recurrent process, for which the one-dimensional marginals are Beta-distributed.

We also investigate the recurrence of the walk and its scaling limit, as well as the ergodicity and mixing of the $n$th step of the walk.
\end{abstract}
\maketitle

\tableofcontents
\section{Introduction}
We start with reviewing the notion of the  {\it coin turning process}, which has been introduced recently in~\cite{EV2018}.

Let $p_2,p_3, p_4...$  be a given deterministic sequence of numbers such that $p_n\in[0,1]$ for all $n$; define also $q_n:=1-p_n$. We define the following time-dependent ``coin turning process" $X_n\in\{0,1\}$, $n\ge 1$, as follows. Let $X_1=1$ (``heads") or $=0$ (``tails") with probability $1/2$. For $n\ge 2$,  set recursively
$$
X_n:=\begin{cases} 
1-X_{n-1},&\text{with probability } p_n;\\
X_{n-1},&\text{otherwise},
\end{cases}
$$
that is, we turn the coin over with probability $p_n$ and do nothing with probability~$q_n$. (Equivalently, one can define $p_1=1/2$ and $X_1\equiv 0$.)

Consider $\overline X_N:= \frac1N \sum_{n=1}^N X_n$, that is, the  empirical frequency of $1$'s (``heads") in the sequence of $X_n$'s. We are interested in the asymptotic behavior of this random variable when $N\to\infty$. Since we are interested in limit theorems, it is convenient to consider a centered version of the variable $X_n$;  namely $Y_n:=2X_n-1\in\{-1,+1\}$. We have then  
$$
Y_n:=\begin{cases}
-Y_{n-1},&\text{with probability } p_n;\\
Y_{n-1},&\text{otherwise}.
\end{cases}
$$
Let  also $\F_n:=\sigma(Y_1,Y_2,...,Y_n),\ n\ge 1$.

Note that the sequence $\{Y_n\}$ can be defined equivalently as follows: 
$$
Y_n:= (-1)^{\sum_{i=1}^n W_i},
$$ 
where $W_1,W_2,W_3,...$ are independent Bernoulli variables with parameters $p_1,p_2,p_3,...$, respectively, and $p_1=1/2$. 
\begin{remark}[Poisson binomial random variable]
The number of turns that occurred up to $n$, that is $\sum_{i=2}^n W_i$, 
is a {\it Poisson binomial random variable}.$\hfill\diamond$ 
\end{remark}


For the centered variables $Y_n$, we have $Y_j=Y_i(-1)^{\sum_{i+1}^jW_k},\ j>i$, and so, using $\Corr$ and $\Cov$ for correlation and covariance, respectively, one has
\begin{align}\label{eq_eij}
 \Corr(Y_i,Y_j)&=\Cov(Y_i,Y_j)=\E( Y_iY_j)=\E(-1)^{\sum_{i+1}^jW_k}\\
&\ \ \ \ =\prod_{i+1}^j \E (-1)^{W_{k}}=\prod_{k=i+1}^j (1-2p_k)
=:e_{i,j};\nonumber\\
 \E (Y_j\mid Y_i)&=Y_i \E(-1)^{\sum_{i+1}^jW_k}=e_{i,j}Y_i.\label{onemore}
\end{align}
The quantity $e_{i,j}$ will play an important role throughout the paper.

Next, we define our  basic object of interest.
\begin{definition}[Coin-turning walk]
The random walk $S$ on $\mathbb Z$ corresponding to the coin-turning, will be called the \emph {coin-turning walk}. Formally, $S_n:=Y_1+...+Y_n$ for $n\ge 1$; we can additionally define $S_0:=0$, so the first step is to the right  or to the left with equal probabilities. As usual, we then can extend $S$ to a {\em continuous time process}, by linear interpolation.
\end{definition}
\begin{remark}
Even though $Y$ is Markovian, $S$ is not. However, the 2-dimensional process $U$ defined by $U_n:=(S_{n},S_{n+1})$ is Markovian. It lives on a ladder embedded into $\mathbb Z^2$. See Figure~\ref{ladder}.$\hfill\diamond$
\end{remark}

\begin{figure}\label{ladder}
  \centering
\setlength{\unitlength}{9mm}
\begin{picture}(11,11)(0,9)
\put(0,9){\vector(0,1){9}}
\put(0,9){\vector(1,0){9}}
\multiput(0,10)(1,1){8}{\vector(1,-1){0.9}}
\multiput(1.1,9)(1,1){8}{\vector(-1,1){0.9}}
\multiput(0,10)(1,1){7}{\vector(1,1){0.9}}
\multiput(2,10)(1,1){7}{\vector(-1,-1){0.9}}
\put(0.5,17){$S_{n+1}$}
\put(8,9.5){$S_{n}$}
\end{picture}
\caption{The process of ordered pairs $U_n:=(S_n,S_{n+1})$ is a Markov chain}
\end{figure}
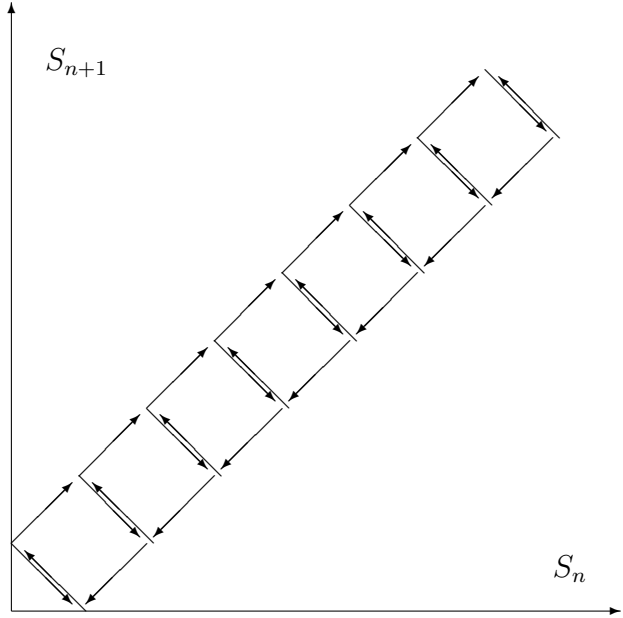

In~\cite{EV2018}, several scaling limits of the form $\lim_{n\to\infty}\mathtt{Law} \left(\frac{S_n}{b_{n}}\right)=\mathtt{L}$, have been established, where $\{b_n\}_{n\ge 1}$ is an appropriate sequence (depending on the sequence of $p_n$'s) tending to infinity and $\mathtt L$ is a non-degenerate  probability law. In~\cite{EV2018} the focus was on the $\lim_{n\to\infty} p_n=0$ case.

\begin{remark}[Supercritical cases] Note that if $\sum_n p_n<\infty$ then by the Borel-Cantelli lemma, only finitely many turns will occur a.s.; therefore the $X_j$'s will eventually become all ones or all zeros, and hence
$$
\overline X_N\to \zeta\text{ a.s.},
$$
where $\zeta\in\{0,1\}$. By the symmetry of the definition with respect to heads and tails
(or, by the bounded convergence theorem), $\zeta$ is a {\sf Bernoulli}($1/2$) random variable.

Similarly, if $\sum_n q_n<\infty$, then $S$ will be eventually stuck at two neighboring integers, again, by the Borel-Cantelli lemma. $\hfill\diamond$
\end{remark}
These two trivial cases (we call them ``lower supercritical"  and  ``upper-supercritical" cases) are not considered, and so we have the following assumption.
\begin{assumption}[Divergence]\label{div.ass} In the sequel we are going to assume that $\sum_n p_n=\infty$ and also $\sum_n q_n=\infty$.
\end{assumption}

\section{Mixing} Unlike in \cite{EV2018} and in the previous section, we now {\it do not} randomize the walk with taking $Y_1$ to be a symmetric random variable.  Nevertheless, it is still true for the indicators of turns $W_k$, that $Y_j=Y_i(-1)^{\sum_{i+1}^jW_k},\ j>i$, and that for $e_{i,j}=\prod_{k=i+1}^j (1-2p_k)$ we have $\E (Y_j\mid Y_i)=Y_i \E(-1)^{\sum_{i+1}^jW_k}=e_{i,j}Y_i,$ hence $\E(Y_iY_j)=e_{i,j}$.
\subsection{Characterization of mixing}
We will say that the sequence of random variables $(Y_n)_{n\ge 1}$ satisfies {\bf the mixing condition} if 
\begin{equation}\label{mixing.eij}
\lim_{j\to\infty} e_{ij}=0, \forall i\in \N.
\end{equation}
Under mixing,  $\lim_{j\to\infty}\E (Y_j\mid Y_i)=0$, so $Y_j$ ``becomes symmetrized" for $i$ fixed and large $j$. Also, $\lim_{j\to\infty}\E(Y_i\, Y_j)=0$ and $\lim_{j\to\infty}\E{Y_j}=0$, hence 
\begin{equation}\label{mixing}
\lim_{j\to\infty} \Cov(Y_j,Y_i)=0,
\end{equation} in accordance with the usual notion of mixing.

Mixing has a very simple characterization in terms of the sequence $\{p_n\}_{n\ge 1}$.
\begin{proposition}[Condition for mixing]\label{mixing.prop}
Mixing holds if and only if 
\begin{equation}\label{mixing.p}
\sum_n \min(p_n,q_n)=\infty.
\end{equation}
\end{proposition}
\begin{proof}[Proof of Proposition~\ref{mixing.prop}]
Since $$\min(p_i,q_i)=\begin{cases}
p_i,&\text{if }p_i\le 1/2;\\
q_i=1-p_i,&\text{if }p_i>1/2,
\end{cases}
$$
we have
\begin{align*}
&|e_{i,j}|=\left|\prod_{k=i+1}^j (1-2p_k)\right|=\prod_{i<k\le j,p_k\le 1/2}(1-2p_k)
\times
\prod_{i<k\le j,p_k>1/2} (1-2q_k)=\\
&\prod_{k=i+1}^j (1-2\min(p_k,q_k)).
\end{align*}

When $p_k\neq 1/2$ for all $k\ge 1$,  \eqref{mixing.eij} and~\eqref{mixing.p} are equivalent by a well known result about infinite products; when $p_k=1/2$ infinitely often, \eqref{mixing.eij} and~\eqref{mixing.p} are clearly simultaneously satisfied.

In all other cases, define $k_0:=\max\{k\in \N\mid  p_k= 1/2\}$. For $i< k_0$, $e_{i,j}=0$ for all large $j$, while for $i\ge k_0$, \eqref{mixing.eij} is tantamount to~\eqref{mixing.p}, just like in the first case.
\end{proof}
\subsection{Why is mixing a natural assumption?}
The mixing condition is stronger than Assumption~\ref{div.ass} if $p_k$ keeps crossing the line $1/2$ (i.e.  $\liminf p_k<1/2<\limsup p_k$), while they are equivalent when $p_k$ settles on one side of $1/2$ eventually.

In the first case Assumption \ref{div.ass} is automatically satisfied, as $p_k\ge 1/2$ i.o. and also $q_k\ge 1/2$ i.o. Defining $I:=\{i\in \N:\ p_i\le 1/2\}$, we see that the mixing condition is nevertheless violated  if and only if
$$
\sum_i \min(p_i,q_i)=\sum_{i\in I} p_i+\sum_{i\not \in I} q_i<\infty,
$$
that is, when $\sum_{i\in I} p_i<\infty$ and $\sum_{i\not \in I} q_i<\infty$. 
In this case, recalling  that $W_i$ is the indicator of a turn at time $i$, by Borel-Cantelli,
$$
\P\left( \exists n_0\in \N:
\ W_{i}=\mathbf{1}_{I^c}(i)  \text{ for all }i\ge n_0 \mid \F_1\right)=1,
$$
where $\mathbf{1}_{I^c}$ is the characteristic function of the set $\N\setminus I$. That is, along $I$, ``turning" eventually stops, while along $\N\setminus I$, ``staying" eventually stops. 

Our conclusion is that when mixing does not hold, the random walk is ``eventually deterministic", and thus the setup is less interesting. For example, from the point of view of recurrence, the problem becomes a question about a deterministic process; whether that process takes any integer value infinitely many times depends simply on the set $I$ (as long as $\ds\sum_{i\in I} p_i<\infty$ and $\ds\sum_{i\not \in I} q_i<\infty$.)

To have a concrete example, let $I=\{2,4,6,\dots\}$ be the set of positive even integers. Then, for large times, the walk will alternate between taking two consecutive steps up and taking two consecutive steps down. This excludes recurrence of course, as the process becomes stuck at some triple of consecutive integers. We summarize the above discussion in Figure \ref{fig.mixing}.

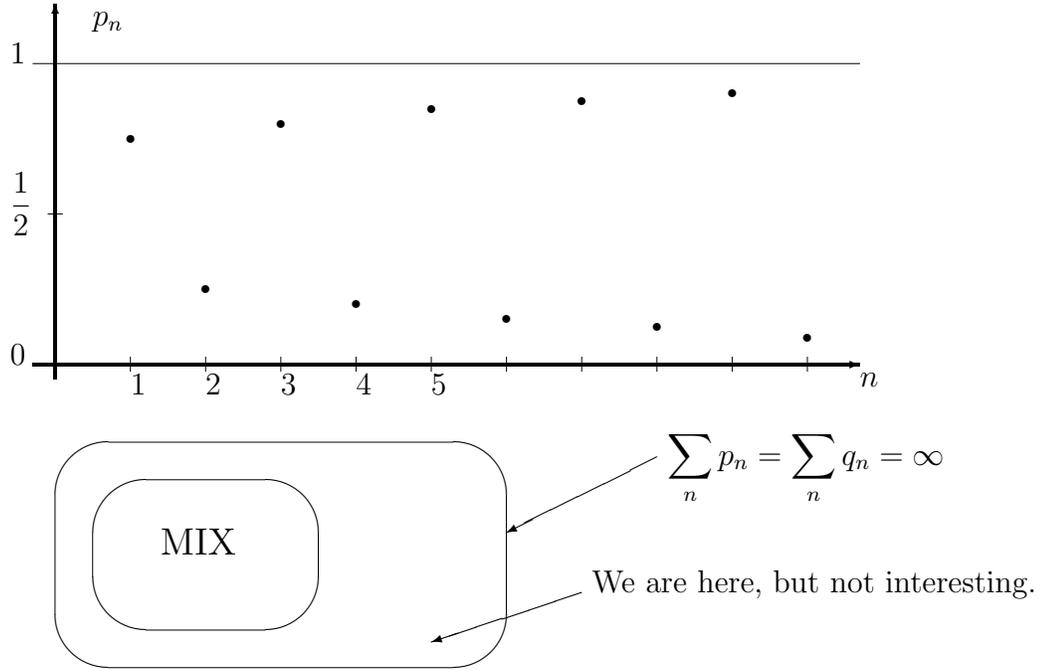
\begin{figure}
\centering
\setlength{\unitlength}{10mm}
\begin{picture}(12,5)(0,0.5)
\linethickness{0.4mm}
\put(1,0.8){\vector(0,1){5}}
\put(0.7,1){\vector(1,0){11}}
\linethickness{0.1mm}
\put(0.7,5){\line(1,0){11}}
\put(0.9,3){\line(1,0){0.2}}
\multiput(2,0.9)(1,0){10}{\line(0,1){0.2}}
\put(2,0.6){$1$}
\put(3,0.6){$2$}
\put(4,0.6){$3$}
\put(5,0.6){$4$}
\put(6,0.6){$5$}
\multiput(2,4)(2,0.2){3}{\circle*{0.1}}
\multiput(3,2)(2,-0.2){3}{\circle*{0.1}}
\multiput(8,4.5)(2,0.1){2}{\circle*{0.1}}
\multiput(9,1.5)(2,-0.15){2}{\circle*{0.1}}
\put(11.7,0.7){$n$}
\put(1.5,5.5){$p_n$}
\put(0.4,1){$0$}
\put(0.4,5){$1$}
\put(0.4,3){$\displaystyle \frac 12$}
\end{picture}

\setlength{\unitlength}{10mm}
\begin{picture}(12,4)(0,0)
\linethickness{0.1mm}

\put(4,2){\oval(6,3)}
\put(3,2){\oval(3,2)}

\put(9,3.2){\ $\displaystyle \sum_n p_n=\sum_n q_n=\infty$}
\put(9,3.3){\vector(-2,-1){2}}

\put(8,1.5){\ We are here, but not interesting.}
\put(8,1.5){\vector(-3,-1){2}}

\put(2.4,2){{\large MIX}}
\end{picture}
\caption{Even if $\sum_n p_n=\sum_n q_n=\infty$ holds, mixing  may fail, as it is equivalent to $\sum_n \min(p_n,q_n)<\infty$.}

\label{fig.mixing}
\end{figure}

We conclude this Section with some notation.

Throughout the paper, $c_n\sim d_n$ will mean that $\lim_{n\to\infty} c_n/d_n= 1$;  also, $c_n=o(d_n)$ will mean that $\lim_{n\to\infty} c_n/d_n= 0$.
Convergence in distribution will also be denoted by $\overset{d}{\rightarrow}$.  

\section{Review of relevant literature}
\subsection{Some  results from~\cite{EV2018}}
Some of the basic results of~\cite{EV2018} are summarized in the following theorem.
\begin{thma}\label{Old.thms} Let $S$ denote the coin-turning walk.
\begin{itemize}
\item[(i)]
(``Time-homogeneous case".) Let $p_n=c$ for all $n\ge 1$, where $0<c<1$. 
Then 
$$
{\sf Law}\left(\frac{S_N}{\sqrt{N}}\right)\to {\sf Normal}\left(0,\sigma_c^2\right),\quad where\ 
\sigma^2_c:=1+2\sum_{i=1}^{\infty} \Cov (Y_i,Y_j)
=\frac{1-c}{c}.
$$
\item[(ii)]
(``Lower critical case".) Fix $a>0$ and let
$$
p_n=\frac{a}n,\quad n\ge n_0
$$
for some $n_0\in \mathbb N$. Then\footnote{A nice exercise, left to the reader, is to show that when the sequence is precisely $(p_1=1/2), p_2=1/3,p_3=1/4,p_4=1/5$, ... , $\frac{S_N}{N}$ has precisely discrete uniform law for each $N$. This fact, as M\'arton Bal\'azs pointed out to us,  can be related to P\'olya urns. The more general connection of our model with P\'olya urns will be presented in a forthcoming paper.}
$$
{\sf Law}\left(\frac{S_N}{N}\right)\to {\sf Beta}(a,a),
$$
where ${\sf Beta}(\a, \b)$ denotes the Beta distribution with parameters $\a,\b$.
\item[(iii)]
(``Lower subcritical case".)
Fix $\g,a>0$ and let
$$
p_n=\frac{a}{n^\g},\quad n\ge n_0
$$
for some $n_0\in\N$. (Since  $\g>1$ corresponds to the supercritical case, we assume that $0<\g<1$.) Then
$$
{\sf Law}\left(\frac{S_N}{\sqrt{N^{1+\g}}}\right)\to {\sf Normal}\left(0,\sigma^2_{a,\g}\right),\quad where \ 
\sigma^2_{a,\g}:=\frac 1{a(1+\g)}.
$$
\end{itemize}
\end{thma}

\subsection{Recent results by Bena\"{\i}m, Bouguet and Cloez} 
In a recent follow up paper to~\cite{EV2018} by Bouguet and Cloez~\cite{BC2018},  the setting has been generalized in such a way that instead of two states (heads and tails or $\pm 1$), one considers $D\ge 2$ states, and with probability $p_n$ in the $n$th step the state changes according to a given irreducible Markov chain.\footnote{E.g. when $D=2$, one can still consider unequal probabilities for switching between the states in different directions.} (They also allow a small error term.) They assume that $\{p_n\}_{n\ge 1}$ is a {\it decreasing} sequence and $p:=\lim_n p_n$ is not necessarily zero. This excludes the $p=1$ case we consider, except the trivial $p_n\equiv 1$ case, and the most interesting case is $p=0$, the one we call cooling dynamics.

Bouguet and Cloez prove several interesting results, generalizing/strengthening those in~\cite{EV2018}. For example they show that if $\sum_n p_n=\infty,\ \lim_{n\to\infty} np_n=\infty$ and $\sum_n (p_n n^2)^{-1}<\infty$, then the empirical distribution of the states converges almost surely to the unique invariant probability distribution of the Markov chain.

The relationship with~\cite{EV2018} is explained in 4.2 in~\cite{BC2018}. 


The paper builds on the authors' previous results with M. Bena\"{\i}m in \cite{BBC2017}, and they point out in \cite{BC2018} that 
\begin{quote}
{\em ``In particular, the results we use provide
functional convergence of the rescaled interpolating processes to the auxiliary Markov processes...''
}
\end{quote}
at which point the authors refer to \cite{BBC2017} and another article.
Also, after their Theorem~2.8, treating the critical $p_n=c/n$ case, they note that
\begin{quote}
{\em ``It should be noted that our approach for the study of the long-time behavior of
... also provides functional convergence for some interpolated process ... from which Theorem 2.8 is a straightforward consequence.''
}
\end{quote}
On the one hand, their Theorem 2.8 is really about the convergence of $S_n/n$ only,  and the ``interpolating process'' alluded to is not the random walk $S$, and it is not completely clear  
if the authors of  \cite{BC2018} are trying to say that one in fact can obtain from~\cite{BBC2017} the functional convergence for $S$ in the critical case as stated in our Theorem~\ref{Lower.Main.Thm}(3). 

On the other hand, it seems that this derivation is after all doable, as we explain this briefly below. The reader may safely skip this part though and return to
it only after reading our main results. Indeed, let us suppose we already know tightness and only want to check the convergence of the finite dimensional distributions, that is the existence of the limit (in law)
\begin{align}\label{frenchlimit}
\lim_{n\to\infty} (S_{nt_{1}}/n,S_{nt_{2}}/n,...,S_{nt_{k}}/n),
\end{align}
for some $0< t_1<t_2<...<t_k$. When $p_n=c/n$ for large $n$'s, define $\tau_t:=\sum_1^{\lfloor t\rfloor} p_n$. Define the ``pasting process'' $\widehat X$ by
$$\widehat X(t):=\sum_{n=1}^\infty \frac{S_n}{n} 1_{\tau_n\le t<\tau_{n+1} },\ t\ge 0.$$
After some algebra, the limit in~\eqref{frenchlimit} is equivalent to the following one:
$$
\lim_{t\to\infty} \left(\widehat X(\tau_{tt_{1}})t_1,\widehat X(\tau_{tt_{2}})t_2,...,\widehat X(\tau_{tt_{k}})t_k\right).
$$
Using the fact that $\lim_{t\to\infty}(\tau_t-\log t)$  is a constant, we can rewrite this as 
$$
\lim_{t\to\infty} (\widehat X^{(\tau_{tt_{1}})}(0)\, t_1,\ \widehat X^{(\tau_{tt_{1}})}(\log(t_2/t_1))\, t_2,\ ...,\ \widehat X^{(\tau_{tt_{k}})}(\log(t_k/t_{k-1}))\, t_k),
$$ 
where $\widehat X^{(z)}(t):=\widehat X(t+z)$, $z>0$.
Now, if the limit (in law) $$\lim_{z\to\infty} (\widehat X^{(z)}(s_1),\ \widehat X^{(z)}(s_2),\ ...,\ \widehat X^{(z)}(s_k))$$ is known, we are done, and this latter type of limit  of ``pseudo-trajectories'' is what has been derived in~\cite{BBC2017} under some suitable assumptions.

In summary, \cite{BC2018} provides a very valuable complement to~\cite{EV2018}. Moreover,  with some further efforts, our result on the zigzag process in the present paper can apparently  be recovered from the results presented in the sequence \cite{BBC2017,BC2018}, and vice versa.

\section{Our main results}
\subsection{The law of the $n$th step for large $n$} Recall that  $$Y_n:= (-1)^{\sum_{i=1}^n W_i},$$ where $W_1,W_2,W_3,...$ are independent Bernoulli variables with parameters $p_1,p_2,p_3,...$, respectively. 

When $p_k\le 1/2$ for all large $k$, $\rho:=\prod_{i=2}^{\infty}(1-2p_i)$ is well defined as the terms are in $[0,1]$ with finitely many exceptions. In particular, when $\sum p_i<\infty$, by Borel-Cantelli, $Y_i=Y$ for all large $i$, a.s., and in Proposition~1  in~\cite{EV2018} it has been shown that in this case
\begin{align*}
\P(Y=1\mid Y_1= 1)&=\lim_n\P(Y_n=1\mid Y_1= 1)=\frac{1+\rho}{2},
\\
\P(Y=1\mid Y_1=- 1)&=\lim_n\P(Y_n=1\mid Y_1=- 1)=\frac{1-\rho}{2}.
\end{align*}
This may be  generalized is as follows.
\begin{theorem}\label{ergod}  
Define $N:=\mathrm{card}\{i: p_i>1/2\}\in \mathbb N\cup \{\infty\}$.
\begin{itemize}
 \item[(a)] If mixing holds, or if $\exists i: p_i=1/2$ then $\lim_n\P(Y_n=1\mid \F_1)=1/2$.
\item[(b)] If mixing does not hold, then there are two cases ($k\in\{-1,1\}$):
 \begin{itemize}
 \item[(i)]  if $N<\infty$, then $\ds\lim_{n\to\infty}\P(Y_n=1\mid Y_1=k)=  \frac12(1+k\rho)$,  and $\rho\neq 0.$
 \item[(ii)] if $N=\infty$ then $\P(Y_n=1\mid Y_1=k)$ has no limit.
 \end{itemize}
 \end{itemize}
\end{theorem} 

\begin{remark}[Ergodicity]\rm
Part (a) in Theorem \ref{ergod} is interpreted as ``mixing implies ergodicity", since $(1/2,1/2)$ is the invariant distribution for the switching matrix  
\[
   M=
  \left( {\begin{array}{cc}
   0 & 1 \\
   1 & 0 \\
  \end{array} } \right),
\] and we can consider our model as one where at step $n$ the  transition given by $M$ may or may not apply (with probabilities $p_n$ and $1-p_n$, resp.).$\hfill\diamond$

\end{remark}
\begin{remark}[Speed of convergence]\rm
Regarding Theorem \ref{ergod}, one may wonder what the speed of convergence is when the limit exists.
A closer look at the proof of Theorem~\ref{ergod} in Section~\ref{Pf} shows that the total variation distance to the limit decays as $$\exp \left(-2\sum_1^n \min\{p_i,1-p_i\}\right).$$ We leave the details to the reader.
$\hfill\diamond$
\end{remark}

\subsection{Scaling limits for the walk}
Recently, Sean O'Rourke has asked whether the results of~\cite{EV2018} could be extended to convergence in the process sense, in the spirit of the classical Donsker invariance principle (see e.g. \cite{KS1991} for the classical result and its proof). We are now going to answer this question, and moreover,  we are also going to consider additional cases, when turns are becoming more and more frequent (i.e.~$p_n$ is close to one), such as, for example, $p_n=1-c/n$ or $p_n=1-n^{-\g}$, $0<\g<1$ for large $n$.

{\bf Note:} In the rest of the paper, for convenience we assume again that $p_1=1/2$, i.e. we symmetrize the setting.
\subsubsection{The time-homogeneous case}
As a warm up, we start with the time-homogeneous case.

\begin{theorem}[Time-homogeneous case:]\label{warmup} Assume that $p_n=c$ for $n\ge n_0$. For $n\ge 1$, define the rescaled walk ${S}^n$  by $${S}^n(t):=\frac{S_{\left\lfloor \frac{c}{1-c}nt
\right\rfloor}}{\sqrt{n}},\ t\ge 0,
$$ 
and let $\mathcal{W}$ denote the Wiener measure. Then $\lim_{n\to\infty}\mathsf{Law}({S}^n)=\mathcal{W}$ on $C[0,\infty)$.
\end{theorem}

\begin{remark}
We will show that Theorem \ref{warmup} follows trivially from our general martingale approximation method of
Subsection \ref{mgale.appr}. However, we note that one can also  give a  direct proof using that the ``turning times" are geometrically distributed. Here is a sketch: assuming that e.g.  $Y_1=1$ we can consider the period consisting of the first run of $1$'s together with the first run of $-1$'s. The second, third etc. periods are defined similarly, and the piece-wise linear ``roof-like" processes in these periods are i.i.d. (up to their respective starting values). Since the length of each run is geometrically distributed, and those geometric variables are independent, the Renewal Theorem applies to the lengths of the periods. One then applies the classical invariance principle to the process considered at each second ``turning time", and finally extends the result for all times. We leave the details to the reader.$\hfill\diamond$
\end{remark}

\begin{remark}
Theorem \ref{warmup} is also covered by those in~ \cite{Da1970,Da1973}. The first one treats the ``uniformly strong mixing" condition for Markov chains and weak convergence.$\hfill\diamond$
\end{remark}

\subsubsection{Heating regime}
The following theorem will give an invariance principle for the ``heating" case, that is for the case when the $p_n$ are getting close to one. But before that we present an important remark.

\begin{remark}[Even and odd parts]
It turns out that in the heating regime, the right approach is to look at the sums of the two sub-series $I=\sum_{\textsl{odd}}:=\ds\sum_{k=1}^{\infty} q_{2k-1}$ and $II=\sum_{\textsl{even}}:=\ds\sum_{k=1}^{\infty} q_{2k}$ separately. If either $I<\infty$ or $II<\infty$, then the invariance principle breaks down. 

Indeed, by Borel-Cantelli then, after some finite time, every other step turns the coin a.s., and consequently, $S$ is  stuck on a set of size three, which rules out the validity of any invariance principle.  We conclude that for an invariance principle to hold, it is not enough to assume merely that $\ds\sum_{k=0}^{\infty} q_k=\infty$; one needs to assume that in fact $I=II=\infty$. $\hfill\diamond$
\end{remark}
In light of the previous remark, without the loss of generality from now on we will work under the following assumption.

\begin{assumption}\label{assA}
$I=II=\infty.$
\end{assumption} 

Before we state the following theorem, we need some more notations,
Introduce
\begin{align}\label{an_defined}
a_n&:=\sum_{i=0}^{\infty}\Cov(Y_{n},Y_{n+i})
=1+\sum_{i=1}^{\infty}(1-2p_{n+1})(1-2p_{n+2})\dots (1-2p_{n+i}),
\\ \nonumber
&=1+\sum_{i=1}^{\infty}(-1)^{i}(1-2q_{n+1})(1-2q_{n+2})\dots (1-2q_{n+i}),\ n\ge 1
\end{align}
which is well defined as the sum of a Leibniz series,
\begin{align}\label{vn_defined}
v_m=\sum_{i=1}^m 4 a_i^2 p_i q_i,\ m\ge 1,
\end{align}
and
\begin{align*}
\xi_i&:=(-1)^{W_i}-\E\left[(-1)^{W_i}\right]=(-1)^{W_i}+2p_i-1;\qquad
\Lambda_n^2:=\sum_{i=1}^n a_i^2 \xi_i^2,
\end{align*} 
so that $\E\xi_i^2=\Var((-1)^{W_i})=4p_i q_i$ and $\E\Lambda^2_n=v_n$.

\begin{theorem}[Invariance principle; heating regime]\label{invar} 
Assume that $q_n\to 0$. Besides Assumption~\ref{assA}, assume that there exists a $C>0$ such that at least one of the following two assumptions is satisfied:
\begin{equation}\label{evens.dominate}
q_{2m}\ge C\max_{\ell\ge m} q_{2\ell+1},\quad \forall m\ge m_0\ (\text{even terms ``dominate''});
\end{equation}
\begin{equation}\label{odds.dominate}
q_{2m+1}\ge C\max_{\ell\ge m+1}q_{2\ell},\quad \forall m\ge m_0\ (\text{odd terms ``dominate''}).
\end{equation}
\begin{itemize}
\item[{\sf(a)}]
Define the rescaled walk ${S}^n$  by setting
\begin{align}\label{def.rescaled.RW}
{S}^n(t):=\frac{S_{Z(nt)}}{\sqrt{n}},\ t\ge 0.
\end{align}
where $Z(x):=\inf\{n\in \mathbb{N}:\  v_n\ge x\}$.
Then 
\begin{align}\label{Wiener.limit}
\lim_{n\to\infty}\mathsf{Law}({S}^n)=\mathcal{W}\ \text{on}\ C[0,\infty),
\end{align}
where $\mathcal{W}$ is the Wiener measure.

\item[{\sf(b)}]
We have $\lim_{n\to\infty}\Lambda_n=\infty$ almost surely\footnote{Note that Drogin in~\cite{D1972} proves, in fact, {\em two} invariance principles. The second one uses the function $s^2$ (our $\Lambda^2$) for time-change.}, and $\ds \lim_{n\to\infty}\frac{\Lambda^2_n}{\E\Lambda^2_n}= 1\ \text{in probability}.$
\end{itemize}
\end{theorem}

\begin{remark}[Equivalent condition] One can rewrite \eqref{evens.dominate} in a ``backward looking" way:
$$
q_{2m+1}\le \mathrm{const}\cdot \min\{q_{2\ell},  \ell\le m\},\ \forall m\ge 0,
$$
as both are equivalent to saying that $q_n\ge \mathrm{const}\cdot q_r$ for $r>n$ if $n$ is even and $r$ is odd. A similar statement holds for \eqref{odds.dominate}.$\hfill\diamond$
\end{remark}

%

\subsubsection{Cooling regime}
When $\displaystyle \lim_{n\to\infty} p_n=0$, one deals with a so-called ``cooling dynamics" as the turns become infrequent. In this case, the scaling limit is not necessarily  Brownian motion, as the following theorem shows. Loosely speaking, the order const$\cdot n^{-1}$ is the critical one in the sense that for sequences of larger order the invariance principle is in force, however at this 
order or below it the situation is dramatically different.

\begin{theorem}[Cooling  regime]\label{Lower.Main.Thm}
Let the process $S^n$ be defined by $S^n(t):=S_{nt}/n,\, t\ge 0$, where for non-integer values of $nt$ we assign $S_{nt}$ using the usual linear interpolation. Let $\mathcal{R}$ be the process (``random ray") defined by  $\mathcal{R}(t):=tR$, where $R$ is a random variable equal to $\pm1$ with equal probabilities. We have the following limits in the process sense:   
\begin{enumerate} 
\item \underline{Supercritical case:} $\ds\sum_{n=1}^\infty p_n<\infty$. Then 
$\lim_{n\to\infty}\|{S}^n(\cdot)-\mathcal{R}(\cdot)\|_{\infty}=0$ almost surely. 
\item \underline{Strongly critical case:} $p_n=o(1/n)$  but $\ds\sum_{n=1}^\infty p_n=\infty$. 
Then 
$\lim_{n\to\infty}{S}^n(\cdot)=\mathcal{R}(\cdot)$ in law.
\item \underline{Critical case:} $p_n=c/n$ for $n\ge n_0$. Recalling the notion of the zigzag process (defined in Section~\ref{Prep.zigzag}),  $\lim_{n\to\infty} S^{(n)}$ is the zigzag process, where the limit is meant in law.
\item \underline{Subcritical case:} (Cooling but larger order than $1/n$)
Let $p_1=1/2$. Assume that, as $n\to\infty$,
\begin{itemize}
\item[(a)]  $ A_n:=n p_n \uparrow\infty$;
\item[(b)] $p_n\downarrow 0$.
\end{itemize}
Then, for  the rescaled walk~\eqref{def.rescaled.RW} the invariance principle~\eqref{Wiener.limit} holds.
\end{enumerate}
\end{theorem}

\subsubsection{Neither heating nor cooling regime}
The following result generalizes the case when $\lim\limits_{n\rightarrow\infty} p_n=a$ with $0<a<1$, as well as the time-homogeneous case of Theorem~\ref{warmup}: the invariance principle holds as long as the $p_n$ are bounded away from both $0$ and~$1$. 
\begin{theorem}[Invariance principle; neither heating nor cooling regime]\label{invar_nheatcool} 
Assume that
\begin{equation}\label{bounded.away}
0<\liminf_{n\rightarrow\infty} p_n\leq \limsup_{n\rightarrow\infty} p_n<1.
\end{equation}
Then for the rescaled walk~\eqref{def.rescaled.RW} the  invariance principle~\eqref{Wiener.limit} holds.
\end{theorem}

\subsection{Validity of the WLLN}
With regard to the Weak Law of Large Numbers (by which we mean that $S_n/n\to 0$ in probability), we know that it breaks down at the critical regime. On the other hand, the following result shows that above that order it is always in force.
\begin{theorem}[WLLN]\label{WLLN.by.comparison} 
Let $p_n\le 1/2$ for all $n\ge 1$ and assume that $\lim_{n\to\infty}np_n=\infty$. Then $\ds \lim_{n\to\infty}\frac{S_{n}}{n}=0$ in probability.
\end{theorem}
\subsection{Recurrence}
We now turn our attention to the recurrence/transience of the walk and its scaling limit.
\begin{definition}
We call $S$ recurrent if
\begin{equation}
\P(S_n=0\ \mathrm{i.o.}\mid Y_1)= 1.
\end{equation}
\end{definition}

Let us introduce the following mild condition on the walk.
\begin{assumption}[Spreading]\label{spr} Assume that for all $n, K\in \N$,
$$\lim_{m\to\infty} \P(|S_{m}|\le K\mid \mathcal{F}_n)=0,\ \mathrm{a.s.}$$
\end{assumption}
\begin{remark}
Assumption \ref{spr} is trivially satisfied when   $\sigma_n^2:=\sf{Var} (S_n)\to\infty$ and the scaling limit 
\begin{equation}\label{sc.lim}
\lim_{m\to\infty} \P\left( \frac{S_{n+m}}{\sigma_{n+m}}\in [a,b]\mid \mathcal{F}_n\right)=Q([a,b]),\ \mathrm{a.s.}
\end{equation}
holds with $a,b\in\R$, $a\le b$ and $n\in\N$, and some probability measure $Q$ such that $Q(\{0\})=0$. These scaling limits we did establish in many cases in \cite{EV2018}.

Let us now assume also mixing. Reformulate \eqref{sc.lim} as
$$
\lim_m \P\left( \frac{S_{n+m}-S_n}{\sigma_{n+m}}\in [a,b]\mid S_n, Y_n\right)=Q([a,b]),\ \mathrm{a.s.}
$$
It is easy to see that the conditioning on $Y_n$ could be safely dropped, as the ``initial'' $n$th step gets forgotten.$\hfill\diamond$
\end{remark}
\begin{theorem}\label{recmix} Besides Assumption~\ref{spr}, assume also mixing. Then $S$ is recurrent.
\end{theorem}
In  the next statement, the part that concerns the walk is a particular case of Theorem \ref{recmix}, provided that one knows that Assumption \ref{spr} holds. (For example, this is the case when $p_n= c/n$ for $n\ge n_0$ with some $n_0$ and $c>0$.)
\begin{theorem}[Recurrence; lower critical case]\label{rec.thm}
Suppose that $p_n\le c/n$ for $n\ge n_0$ with some $n_0$ and $c>0$, and at the same time $\sum_n p_n=\infty$. Then $S$ is recurrent, and in the $p_n= c/n$, $n\ge n_0$ case, the scaling limit (zigzag process) is recurrent as well. 
\end{theorem}
Finally, we would like to summarize our scaling results in the diagram on Figure~\ref{fig.regimes}.

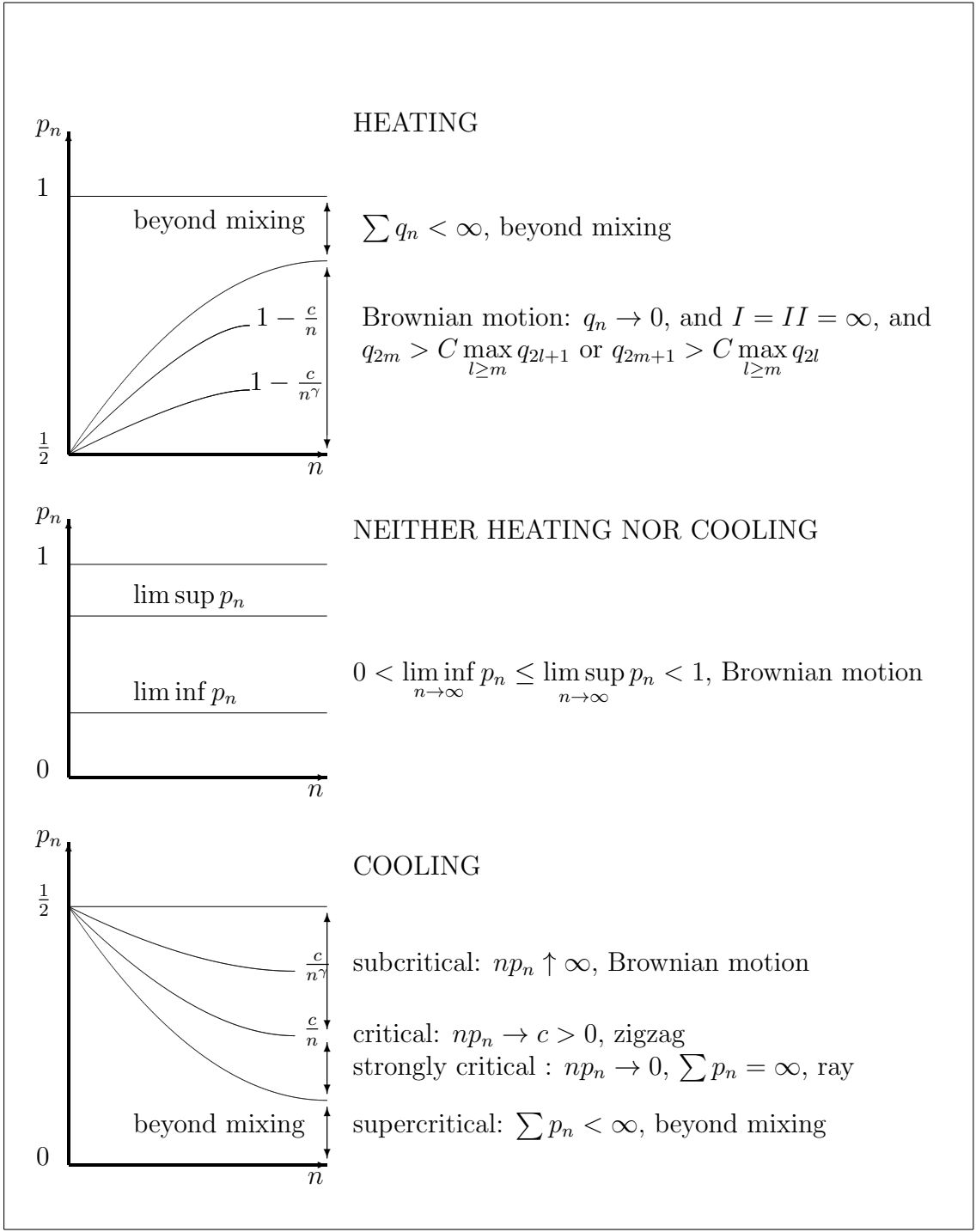
\begin{figure}
\centering
\setlength{\unitlength}{10mm}
\begin{picture}(15,19)
\put(0,0){\line(1,0){15}}
\put(0,0){\line(0,1){19}}
\put(15,19){\line(-1,0){15}}
\put(15,19){\line(0,-1){19}}
\linethickness{0.4mm}
\put(1,1){\vector(0,1){5}}
\put(1,1){\vector(1,0){4}}
\linethickness{0.1mm}
\put(1,5){\line(1,0){4}}
\qbezier(1,5)(3,2)(5,2)
\qbezier(1,5)(3,3)(4.5,3)
\qbezier(1,5)(3,4)(4.5,4)
\put(0.5,1){$0$}
\put(0.5,5){$\frac12$}
\put(0.5,6){$p_n$}
\put(4.5,3){\ $\frac cn$}
\put(4.5,4){\ $\frac c{n^\gamma}$}
\put(5.4,5.5){COOLING}
\put(5.4,4){subcritical: $np_n\uparrow\infty$, Brownian motion}
\put(5.4,2.9){critical: $np_n\to c>0$, zigzag}
\put(5.4,2.4){strongly critical : $np_n\to 0$, $\sum p_n=\infty$, ray}
\put(5.4,1.5){supercritical: $\sum p_n<\infty$, beyond mixing}
\put(5,1.1){\vector(0,1){.8}}
\put(5,1.9){\vector(0,-1){.8}}
\put(5,2.1){\vector(0,1){.8}}
\put(5,2.9){\vector(0,-1){.8}}
\put(5,3.1){\vector(0,1){1.8}}
\put(5,4.9){\vector(0,-1){1.8}}
\linethickness{0.4mm}
\put(1,7){\vector(0,1){4}}
\put(1,7){\vector(1,0){4}}
\linethickness{0.1mm}
\put(1,8  ){\line(1,0){4}}
\put(1,9.5){\line(1,0){4}}
\put(1,10.3){\line(1,0){4}}
\put(0.5,7){$0$}
\put(0.5,10.3){$1$}
\put(0.5,11){$p_n$}
\put(5.4,10.7){NEITHER HEATING NOR COOLING}
\put(5.4,8.5){$\displaystyle 0<\liminf_{n\to\infty} p_n\le \limsup_{n\to\infty} p_n<1$, Brownian motion}
\put(2,8.2){$\liminf p_n$}
\put(2,9.7){$\limsup p_n$}
\linethickness{0.4mm}
\put(1,12){\vector(0,1){5}}
\put(1,12){\vector(1,0){4}}
\linethickness{0.1mm}
\qbezier(1,12)(3,15)(5,15)
\qbezier(1,12)(3,14)(3.8,14)
\qbezier(1,12)(3,13)(3.8,13)
\put(1,16){\line(1,0){4}}
\put(0.5,12){$\frac 12$}
\put(0.5,16){$1$}
\put(0.5,17){$p_n$}
\put(3.9,14){$1-\frac cn$}
\put(3.8,13){$1-\frac c{n^\gamma}$}
\put(5.4,14){ Brownian motion: $q_n\to 0$, and  $I=II=\infty$,  and}
\put(5.4,13.5){ $\displaystyle q_{2m}>C\max_{l\ge m} q_{2l+1}$ or $\displaystyle q_{2m+1}>C\max_{l\ge m} q_{2l}$}
\put(5.4,15.4){ $\sum q_n<\infty$, beyond mixing}\put(5.4,17){HEATING}
\put(5,12.3){\vector(0,1){2.6}}
\put(5,14.7){\vector(0,-1){2.6}}
\put(5,15.1){\vector(0,1){.8}}
\put(5,15.9){\vector(0,-1){.8}}
\put(4.7,0.7){$n$}
\put(4.7,6.7){$n$}
\put(4.7,11.7){$n$}
\put(2,1.5){beyond mixing}
\put(2,15.5){beyond mixing}
\end{picture}
\caption{Three regimes of possible convergence.}
\label{fig.regimes}
\end{figure}


\section{Examples and open problems}
In this section, we compute the scaling $Z(\cdot)$ for a few examples in the cooling regime and the heating regime. We first give two concrete examples for the heating regime. Notice that the scaling function $Z(\cdot)$ is the generalized inverse of $v(m):=\sum_{n=1}^m 4 a_n^2 p_n q_n$. Hence, it suffices to compute $v(m)$ in order to obtain the scaling of~$S^{(n)}$.
\begin{example}
[Heating regime]  Set $p_n=1-\frac{c}{2\, n^\g}$, for $n\ge n_0$, where $0<\g<1$. By Proposition~\ref{equiv.for.b} in Section~\ref{mgale.appr},  $\Var(S_m)=(1+o(1))v(m)$, so we only need to compute $\Var(S_m)$, and then $Z(\cdot)$ is asymptotically equivalent to the ``inverse" of $\Var(S_m)$. First, note that 
$$
e_{ij}=\Cov(Y_i,Y_j)=\prod_{k=i+1}^j (1-2p_k), \qquad
|e_{ij}|= \prod_{k=i+1}^j \left(1-\frac{c}{k^\g}\right),
$$
and
$$
\Var(S_n)=n+2\sum_{i=1}^{n-1}\sum_{j=i+1}^n e_{ij}
=n+2\sum_{i=1}^{n-1}\sum_{j=i+1}^n (-1)^{i+j}
\prod_{k=i+1}^j |e_{ij}|.
$$
Thus
$$
\Var(S_n)-\Var(S_{n-1})=1+2\sum_{i=1}^{n-1} e_{in}
=1-2\sum_{i=1}^{n-1} (-1)^{n-1-i}|e_{in}|.
$$
Let us now show that
\begin{equation}\label{eq2.0}
\sum_{i=1}^{n-1} (-1)^{n-1-i}|e_{in}|
=
\sum_{i=1}^{n-1} (-1)^{n-1-i}\prod_{k=i}^n \left(1-\frac2{k^\g}\right)=\frac 12 -\frac{c+o(1)}{4n^\g}.
\end{equation}
In the case when $i\le n-n^{\frac{2\g+1}3}$ (note that $\g<\frac{2\g+1}3<1$), one has
$$
|e_{in}|\le \prod_{k=n-n^{\frac{2\g+1}3}}^n \left(1-\frac c{k^\g}\right)
\le \left(1-\frac{c}{n^\g}\right)^{n^{\frac{2\g+1}3}}
\le \exp\left(-cn^{\frac{1-\g}3}\right)
$$ 
yielding
\begin{equation}\label{eg2.1}
\sum_{i=1}^{n-n^{\frac{2\g+1}3}}|e_{in}|< n e^{-cn^{\frac{1-\g}3}} =o(n^{-\g}).
\end{equation}
For $i\ge n-n^{\frac{2\g+1}3}$, we have
\begin{equation}\label{eq2.2}
\begin{aligned}
&\sum_{i=n-n^{\frac{1+2\g}3}}^{n-1} (-1)^{n-1-i}|e_{in}|
=(|e_{n-1,n}|-|e_{n-2,n}|)+(|e_{n-3,n}|-|e_{n-4,n}|)+\dots
\\ &
=\frac c{(n-1)^\g} |e_{n-1,n}|
+\frac c{(n-3)^\g} |e_{n-3,n}|
+\frac c{(n-5)^\g} |e_{n-5,n}|
+\dots\\ &
=d_1+d_3+d_5+\dots
=\sum_{j=1,\mathrm{\ odd}}^{n^{\frac{(2\g+1)}{3}}} d_j,
\end{aligned}
\end{equation}
where
$$
d_j=\frac c{(n-j)^\g} \left(1-\frac{c}{(n-j+1)^\g}\right)
\left(1-\frac{c}{(n-j+2)^{\g}}\right)...
\left(1-\frac{c}{n^{\g}}\right),
$$
with $1\le j\le n^{\frac{2\g+1}3}$. Define also 
$$
b_j:=\kappa(1-\kappa)^j, \quad \text{ where }\kappa=\frac{c}{n^\g}.
$$
Note that
$$
d_j\le\frac{c}{(n-j)^\g} \left(1-\frac{c}{n^{\g}}\right)^j
=\left(1-\frac{j}{n}\right)^{-\g} b_j
=\left(1+O\left(n^{-\frac{2-2\g}3}\right) \right) b_j
$$
but
\begin{align*}
d_j&\ge\frac{c}{n^\g} \left(1-\frac{c}{(n-j+1)^{\g}}\right)^j
=\left(1- c\frac{\left(1-\frac{j-1}n\right)^{-\g}-1}{n^\g-c} \right)^j b_j\\ &
=\left(1-O\left(\frac{j}{n^{1+\g}}\right) \right)^j b_j
=\left(1-O\left(\frac{j^2}{n^{1+\g}}\right) \right) b_j
=\left(1-O\left(n^{-\frac{1-\g}3}\right) \right) b_j.
\end{align*}
Hence,
$$
|b_j-d_j|=b_j\times o(1),
$$
implying
\begin{equation}\label{eq2.3}
\sum_{j=1,\mathrm{\ odd}}^{n^{(2\g+1)/3}} d_j
=(1+o(1))\sum_{j=1,\mathrm{\ odd}}^{n^{(2\g+1)/3}} b_j.
\end{equation}
At the same time,
\begin{align*}
b_1+b_3+\dots&=\kappa(1-\kappa)[1+(1-\kappa)^2+(1-\kappa)^4+\dots]=\frac{\kappa(1-\kappa)}{1-(1-\kappa)^2}= \frac{1-\kappa}{2-\kappa}\\ &
=\frac12-\frac{\kappa}{2(2-\kappa)}=
\frac12 -\frac{c+o(1)}{4n^\g},
\end{align*}
so
\begin{equation}\label{eq2.4}
\sum_{j=1,\mathrm{\ odd}}^{n^{(2\g+1)/3}} b_j=
\sum_{j=1,\mathrm{\ odd}}^{\infty} b_j-O\left((1-\kappa)^{n^{\frac{2\g+1}3}}\right)
=\sum_{j=1,\mathrm{\ odd}}^{\infty} b_j-O\left(e^{-cn^{\frac{1-\g}3}}\right)
=\frac12 -\frac{c+o(1)}{4n^\g}.
\end{equation}
Then, combining \eqref{eg2.1}, \eqref{eq2.2}, \eqref{eq2.3} and \eqref{eq2.4} we obtain \eqref{eq2.0}. Hence
$$
\Var(S_n)-\Var(S_{n-1})=1-2\left[\frac 12 -\frac{c+o(1)}{4n^\g} \right]=\frac{c+o(1)}{2n^\g},
$$
and as a result, $\Var(S_n)=\frac{c}{2(1-\g)}n^{1-\g}+o(n^{1-\g})$.\\
Our conclusion is that  $Z(x)\sim\lfloor (2x(1-\g)/c)^{\frac{1}{1-\g}}\rfloor$, that is, for the rescaled walk \eqref{def.rescaled.RW}
the limit in \eqref{Wiener.limit} holds.
\end{example}
\begin{example}[Heating regime] Let $p_n=1-\frac cn$, $n\ge n_0$, for some $n_0\ge 1$. From Lemma \ref{simplemma} in Section~\ref{Pf}, $\lim_{n\to\infty}a_n= 1/2$, hence 
$$
v_m=\sum_{n=1}^m 4a_n^2 p_nq_n=(1+o(1))\sum_{n=1}^m \left(1-\frac{c}{n}\right)\frac{c}{n}=(c+o(1))\ln n.
$$
Thus, for the rescaled walk \eqref{def.rescaled.RW}, the limit \eqref{Wiener.limit} holds, but now with $Z(x)\sim\lfloor e^{x/c} \rfloor$.
\end{example}
Next is an example for the cooling regime.
\begin{example}[Subcritical case; cooling regime]
If  $p_n=\frac{c}{n^{\g}}$  for some $c>0$, $\g\in(0,1)$ and all $n\ge n_0$, then for the rescaled walk \eqref{def.rescaled.RW}
the  invariance principle \eqref{Wiener.limit} holds. Indeed, similarly to the previous examples, one only needs to know the order of  $\Var(S_m)$. By Theorem 2 of \cite{EV2018}, $\Var(S_n)=(1+o(1))\frac{n^{1+\g}}{c(1+\g)}$, so $Z(x)\sim\left\lfloor[c(1+\g)]^{\frac{1}{1+\g}}(x)^{\frac{1}{1+\g}}\right\rfloor$. 
\end{example}

\medskip
We finally present a few open problems.
\begin{problem}[When $p_n$ is not comparable to $1/n$; different PPP's]\rm
What can be said about the case when $\liminf_ n np_n=0$ and $\limsup_n np_n=\infty$? A somewhat related question is whether the following is possible for some situations: the scaling limit is a piecewise deterministic process and the turning points form a PPP but the intensity is different from const$/x\, \mathrm{d}x$.
 \end{problem}
\begin{problem}[Random temporal environment]\rm
One can also consider a {\it random walk in a random temporal environment} (as opposed to the more usual random spatial environment) as follows. Assume now that the $p_n$ are i.i.d. random and follow the same distribution (supported on $[a,b]$, for $0<a<b<1$) or a family of distributions on $[a,b]$. What can one say about the walk in the quenched or in the annealed case? 
 \end{problem}

\section{Proofs}\label{Pf}
The rest of the paper is organized as follows.
After presenting two preparations sections on martingale approximation and on a piecewise deterministic process, we give the proofs of the main results.

\subsection{Preparation I: The zigzag process}\label{Prep.zigzag}
We now define a stochastic process, which we will relate to the critical case in the cooling regime.
\begin{definition}[Zigzag process]\rm 
Consider a Poisson point process (PPP) on $[0,\infty)$ with intensity measure $\frac{a}{x}\, \mathrm{d}x$ with $a>0$. Once the realization is fixed, the value of the process  at $t\ge 0$ is obtained as follows. Starting with the segment containing $t$ and going backwards towards the origin,  color the first, third, fifth, etc.  segments between the points blue. The second, fourth, etc.\ will be colored red. Given this Poisson intensity, we will have infinitely many segments towards zero (and also towards infinity) almost surely.

\begin{figure}[h!]
\includegraphics[width=10cm]{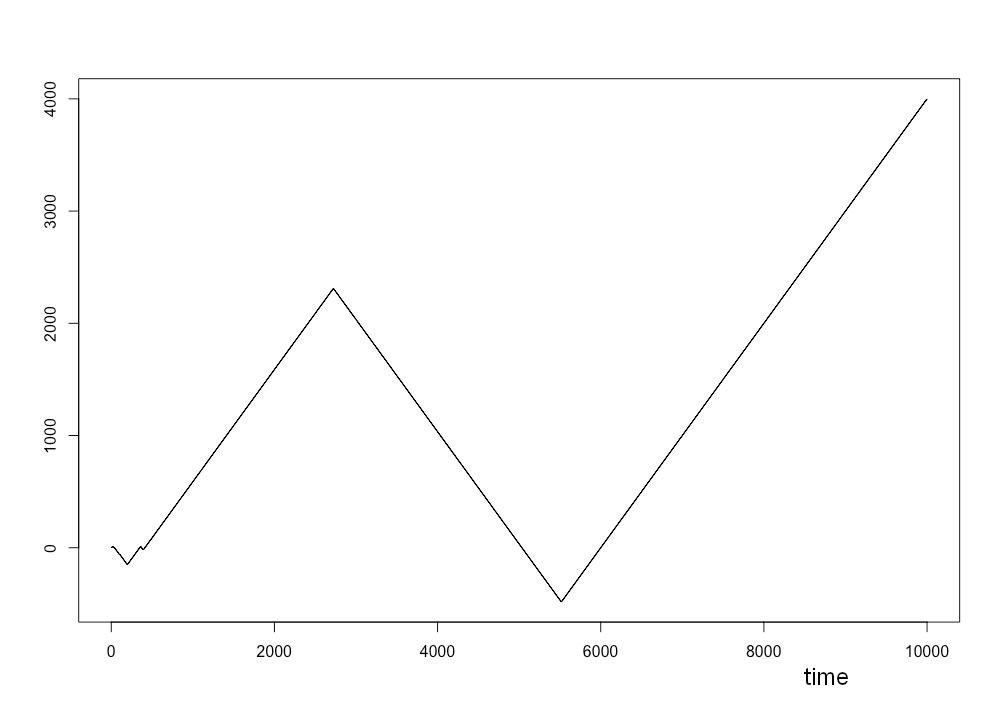}
\caption{The zigzag process: turning points form a PPP on $[0,\infty)$ with intensity measure $\frac{a}{x}\, \mathrm{d}x$. (Obtained by simulating $S$.)}\label{zigzagpic}
\end{figure}

Let $\lambda_b(t)$ and $\lambda_r(t)$ denote the Lebesgue measure of the union of blue, resp.\ red segments between $0$ and $t$. Then we define the \emph {zigzag process} $X$ by $$X_t:=W[\lambda_b(t)-\lambda_r(t)],$$ where $W$ is a random sign, that is $W=-1$ or $W=1$ with equal probabilities. See Figure~\ref{zigzagpic}.
\end{definition}
It is easy to check directly that the law of the process is invariant under scaling both axes by the same number.
\begin{remark}[One-dimensional marginals]\rm  It is more challenging to check directly that the one-dimensional marginals of the zigzag process are $\mathsf{Beta}(a,a)$, although this  follows immediately from Theorem \ref{Lower.Main.Thm} along with the scaling limit result for the one-dimensional distributions in \cite{EV2018}. Edward Crane has shown us a nice direct proof for this fact though. The interested reader may enjoy trying to find such a  proof him/herself.$\hfill\diamond$
\end{remark}
\subsection{Preparation II: Approximating the walk with a martingale}\label{mgale.appr}
We are interested in the scaling limit of the random walk $S$, and in particular, whether we have a Donsker-style invariance principle, leading eventually to Brownian motion. Following the general principle that ``it always helps to find a martingale",  in this section we investigate the following important, though still somewhat vague, question.
\begin{question}[M]
For a given sequence $\{p_n\}_{n\ge 1}$ is the walk $S$  ``sufficiently close" to some martingale $M$?
\end{question}
After Question (M),  the next question is of course:
\begin{question}[INV.M] Is there an invariance principle for $M$?
\end{question}
Focusing now on Question (M) only, we recall from \eqref{eq_eij} and~\eqref{onemore} the identity $e_{i,j}=\E (Y_j\mid Y_i)/Y_i,$ and that for $1\le i<j<k$, $e_{i,j}e_{j,k}=e_{i,k}.$ 
With the convention $e_{i,i}:=\E(Y^2_i)=1$, recall the definition of $a_n=\sum_{i=0}^{\infty}e_{n,n+i}$ from~\eqref{an_defined}, assuming that the series is convergent (if $p_n\ge 1/2$ for large $n$, then it always is; see below). Then 
$$
M_n:=Y_1+\dots Y_{n-1}+a_nY_n
$$
is a martingale. Indeed,
\begin{align*}
 \E(M_{n+1}-M_n\mid\mathcal{F}_n)&=\E((1-a_n)Y_n+a_{n+1}Y_{n+1}\mid\mathcal{F}_n)\\&
 =(1-a_n)Y_n+a_{n+1}\E(Y_{n+1}\mid Y_n)\\
&= [(1-a_n)+a_{n+1}e_{n,n+1}]Y_n,
\end{align*}
which is identically zero, since $a_{n+1}e_{n,n+1}=a_n-1$, as
$$
a_{n+1}e_{n,n+1}=\sum_{i=1}^{\infty}e_{n,n+1}e_{n+1,n+i}=\sum_{i=1}^{\infty}e_{n,n+i}=a_n-1.
$$
Observe also that
\begin{align}\label{VarMn}
\nonumber
\Var(M_{n+1}-M_n)&=\Var\left((1-a_n)Y_n+a_{n+1}Y_{n+1}\right)\\ & 
\nonumber
=(1-a_n)^2\Var(Y_n)+a_{n+1}^2\Var(Y_{n+1})+2(1-a_n)a_{n+1}\Cov(Y_n,Y_{n+1})
\\
& 
=a_{n+1}^2+(1-a_n)^2+2(1-a_n)a_{n+1} e_{n,n+1}
=a_{n+1}^2-(1-a_n)^2
\\ & \nonumber
=a_{n+1}^2\left[1-e_{n,n+1}^2\right]=4a_{n+1}^2 p_{n+1} q_{n+1}
\end{align}
since $\Var(Y_n)=\E(Y_n^2)=1$ for each $n$.

To understand what we mean by being {\it sufficiently close} to a martingale, recall that the rescaled walk ${S}^n$  is defined by 
$$
{S}^n(t):=\frac{S_{Z(nt)}}{\sqrt{n}}=\frac{M_{Z(nt)}+(1-a_{Z(nt)})Y_{Z(nt)}}{\sqrt{n}},\ t\ge 0,
$$ 
where
\begin{align*}
Z(n):=\inf\{m:\ v_m\ge n\},\ n\ge 1,
\end{align*}
Since $|Y_k|=1$, if  the $a_n$ are not too large, then it suffices to analyze the sequence of the rescaled martingales $M^n(t):=\frac{M_{Z(nt)}}{\sqrt{n}}$ instead of the sequence of the rescaled random walks.
Thus, we have the answer in the affirmative to Question (M), provided that  
\begin{enumerate}
\item[\sf (a)] $a_n$ is well-defined;
\item[\sf (b)] $a_{Z(n)}=o(\sqrt{n})$ (e.g. when $a_n$ remains bounded) as $n\to\infty$. (We dropped $t$ as it is just a constant.) 
\end{enumerate}

\begin{proposition}[Equivalent conditions for \sf{(b)}]\label{equiv.for.b}
Set 
$$
\sigma_n^2:= \Var(S_n).
$$ 
Since the martingale differences $M_i-M_{i-1}$ are uncorrelated and centered, one has 
$$
\Var(M_n)=\E\left[\left(\sum_{i=1}^n [M_i-M_{i-1}]\right)^2\right]=\sum_{i=1}^n \E[(M_i-M_{i-1})^2]=v_n,
$$  
where $v_n$ is defined by~\eqref{vn_defined} and $\Var(M_i-M_{i-1})$ is given by~\eqref{VarMn}. Then the conditions
\begin{itemize}
\item[({\sf b.1})] 
$v_n\to\infty,a_n=o\left(\sqrt{v_n}\right)$;
\item[({\sf b.2})]
$\sigma_n\to\infty, a_n=o(\sigma_n)$
\end{itemize}
are equivalent; and when they are satisfied,  $\sqrt{v_n}\sim \sigma_n$. 

Of course, ${\sf (b.1)}\Leftrightarrow {\sf (b.2)}\Rightarrow {\sf(b)}$. Moreover, if $v_n\to\infty$, then the condition $a_n=o\left(\sqrt{v_n}\right)$ is in fact equivalent to {\sf(b)}. The proofs of these statements are provided later.$\hfill\diamond$
\end{proposition}
 
 To answer Question (INV.M), we refer the invariance principle of Drogin.
\begin{propD}[Part of Theorem~1 in~\cite{D1972}]
Let  $(X_i)_{i\ge 1}$ be a sequence of square integrable random variables adapted to the filtration $(\F_i)_{i\ge 1}$. Assume  that they are martingale differences: $\E(X_{i}\mid\F_{i-1})=0$,  and that $v_m:=\sum_{i=1}^m \E(X^2_{i}\mid\F_{i-1})\to\infty$ a.s. The processes $S$  and $S^n,\ n\ge 1$  by $S(v_m)=\sum_{i=1}^m X_i,\  S(0):=0,$ and by $S^n(t):=S(nt)/\sqrt n,\ t\ge 0$, using linear interpolation between integer times. Then the following are equivalent:
\begin{itemize}
\item[(i)] If $\epsilon>0$, then 
\begin{equation}\label{Drogin.first}\frac1n\sum_{i=1}^{Z(n)} X_i^2 1_{\{X_i^2>n\epsilon\}}\to_{L_1} 0\quad\text{as}\ n\to\infty.
\end{equation}
\item[(ii)]
As $n\to\infty$, the law of $S^n$ converges to the Wiener  measure and $$\frac{v_{Z(n)}}{n}\to_{L_1} 1.$$
\end{itemize}
\end{propD}
\medskip

Note that, in our setting, both $v_m$ and $Z(n)$ are deterministic. To summarize the discussion of Question(M) and (INV.M) above, in our setting, once the limit process is Brownian motion, we need to check the following conditions,
\begin{enumerate}
\item[\sf (a)] $a_n$ is well-defined;
\item[\sf (b)] $a_{Z(n)}=o(\sqrt{n})$, or equivalently, $a_{n}=o(\sigma_n)$ (given $v_n\rightarrow \infty$), as $n\to\infty$.
\item[\sf (c)] $v_n\rightarrow \infty$ and \eqref{Drogin.first} holds.
\end{enumerate}
Here (a), (b) guarantee (M) affirmative , and (c) guarantees (IN.V) affirmative.

\subsection{Some specific cases}
The first two cases we are looking at are in the cooling regime, the last one is in the heating regime.
We will use the conditions discussed in the last paragraph in Proposition~\ref{equiv.for.b}.
\subsubsection{Cooling, critical}
Let $p_n=c/n$ for large $n$. If $c\ge 1/2$, then (a) fails to hold, because then $a_n=\infty$. Otherwise $a_n$ is of order  $n^{1-2c}$, and $\sqrt{v_n}$ is of the same order, and thus {\sf(b)} fails to hold. In both cases, the answer to Question (M) is negative.
\subsubsection{Cooling, subcritical} Let $p_n\le 1/2$ for all\footnote{We may assume this without the loss of generality, as the validity of the  invariance principle does not depend on a finite number of terms.} $n\ge 1$ and $p_n=c/n^{\g}$ for $n$ large, where $0<\g<1$. In this case the answers to (M) and to (INV.M) are both in the affirmative, and one can compute that $a_n=\frac{n^{\g}}{2c}(1+o(1))$.
\subsubsection{Cooling, subcritical; the necessity of $\liminf\limits_{n\rightarrow\infty} \frac{p_n}{p_{n+1}}>0$} 
One can see that assumption (a) in Theorem~\ref{Lower.Main.Thm}(4) guarantees that 
\begin{equation*}
\liminf_{n\to\infty} \frac{p_n}{p_{n+1}}>0.
\end{equation*} 
The following example shows indeed  the necessity of this bound, that is, by showing that the property that $a_n=o(v_n)$ can break down if this $liminf$ is zero. Indeed, let
$$
p_i:=\frac{\ln k}{2\cdot k!}\text{\ \ for } k!< i\le (k+1)!,\quad k=1,2,\dots.
$$
Then
$$
\prod_{i=k!+1}^{(k+1)!}(1-2p_i)=\left(1-\frac{\ln k}{k!}\right)^{k\cdot k!}=(1+o(1))e^{-k\ln k}=\frac{1+o(1)}{k^k}
$$
and $\sum_k \frac{(k+1)!-k!}{k^k}<\infty$, so $a_n$ is well-defined. Moreover,
\begin{align*}
a_{m!}&=\sum_{i=0}^{\infty}e_{m!,m!+i}
\ge 1+\sum_{i=1}^{(m+1)!-m!}(1-2p_{m!+1})\dots(1-2p_{m!+i})
\\ &=
1+\left[1-\frac{\ln m}{m!}\right]+\left[1-\frac{\ln m}{m!}\right]^2+\dots+\left[1-\frac{\ln m}{m!}\right]^{(m+1)!-m!}\\ & =
\frac {1-O\left(e^{-m\ln m}\right)}{1-\left(1-\frac{\ln m}{m!}\right)}=(1+o(1))\, \frac{m!}{\ln m}.
\end{align*}
At the same time,
\begin{align*}
\frac{v_{m!}}{4}&=\sum_{i=1}^{m!} a_i^2 p_i q_i =  \sum_{k=0}^{m-1}\sum_{i=k!+1}^{(k+1)!}  a_i^2 p_i q_i
\le 
\sum_{k=0}^{m-1}\left[\sum_{i=k!+1}^{(k+1)!}  a_i^2\, \frac{\ln k}{k!}
\right]
\\
&\le \sum_{k=0}^{m-1} (1+o(1))\frac{k!}{\ln k}\le \frac{(1+o(1))(m-1)!}{\ln (m-1)}\le \frac{1+o(1)}m\cdot \frac{m!}{\ln m}
=o\left(a_{m!}^2\right),
\end{align*}
since for $k!<i\le (k+1)!$,
\begin{align*}
a_{i}&\le
\sum_{j=0}^{(k+1)!-k!} \left[1-\frac{\ln k}{k!}\right]^j
+ 
\left[1-\frac{\ln k}{k!}\right]^{(k+1)!-k!}\cdot \sum_{j=0}^{(k+2)!-(k+1)!} \left[1-\frac{\ln (k+1)}{(k+1)!}\right]^j
\\
&+
\left[1-\frac{\ln k}{k!}\right]^{(k+1)!-k!}\cdot 
\left[1-\frac{\ln (k+1)}{(k+1)!}\right]^{(k+1)!-k!}\cdot 
\sum_{j=0}^{(k+3)!-(k+2)!} \left[1-\frac{\ln (k+2)}{(k+2)!}\right]^j+\dots
\\
&\le (1+o(1))\frac{k!}{\ln k}+e^{-k\ln k}(k+2)!
+e^{-k\ln k}e^{-(k+1)\ln (k+1)}(k+3)!+\dots
\\
&\le(1+o(1))\frac{k!}{\ln k}+\frac{(k+2)!}{k^k}
+\frac{(k+3)!}{(k+1)^{m+1}}+\frac{(k+4)!}{(k+2)^{k+2}}+...
=(1+o(1))\frac{k!}{\ln k}.
\end{align*}

At the same time it is worth noting that with these $p_i$'s, the assumption (4)(a) in Theorem \ref{Lower.Main.Thm} is violated too, since for $i=(m+1)!$,
one has  
\begin{align*}
ip_i= (m+1)!\, p_{(m+1)!}=
(m+1)!\frac{\ln m}{2\cdot m!}=\left[\frac 12+\frac m2\right] \ln m,
\end{align*}
while 
$$
(i+1)p_{i+1}=\left[(m+1)!+1\right]\frac{\ln(m+1)}{2\cdot(m+1)!}
=\left[\frac12+o(1)\right] \ln m\ll i p_i.
$$
\subsubsection{\bf(Heating)}
Let  $p_n=1-q_n$ and $q_n\to 0$ but $\sum q_n=\infty$. We have
$$
a_n=1+\sum_{i=1}^{\infty}(-1)^{i}(1-2q_{n+1})(1-2q_{n+2})\dots (1-2q_{n+i}),
$$
and, since $1-2p_n=2q_n-1<0$ for large $n$, using the Leibniz criterion, along with the assumption that $\sum q_n=\infty$, it follows that $a_n$ is well defined. The validity of the martingale approximation follows from the fact that $a_n\le 1$ but $v_n\to\infty$; see the proof of Theorem~\ref{invar}.\\

\subsection{Proof of Theorem~\ref{ergod}}
Clearly, if $p_i=1/2$ for some $i\in \N$ then the process ``gets symmetrized" from time $i$ on (and $\rho=0$), and the statement is trivial. We will thus assume in the rest of the proof that $p_i\neq 1/2,\ \forall i\in \N$.

Furthermore, we will handle the conditional probability $\P(\cdot \mid Y_1=1)$ only, the argument for $\P(\cdot \mid Y_1=-1)$ is similar. In terms of the $W_i$, one has  $Y_n:= (-1)^{\sum_{i=1}^n W_i},$ where $W_1,W_2,W_3,...$ are independent Bernoulli variables with parameters $p_1,p_2,p_3,...$, respectively and we will handle the $p_1=0$ (i.e. $W_1\equiv 0$) case. In particular,
$\prod_{i=1}^n (1-2p_i)=\prod_{i=2}^n (1-2p_i).$ 

Let $x_n:=\P(Y_n=1)$. We have the recursion 
\begin{eqnarray*}
x_{n+1}&=&p_n(1-x_n)+(1-p_n)x_n,\ n\ge 1;\\
x_1&=&1,
\end{eqnarray*}
and the substitution $y_n:=x_n-1/2$ yields $y_{n+1}=(1-2p_n)y_n$  with $y_1=1/2$. Hence,
\begin{equation}\label{second.rec}
y_{n+1}=\frac12\prod_{i=1}^n (1-2p_i).
\end{equation}

\noindent\underline{Case 1:   $N=\infty$.} We have to prove that $x_n$ converges to 1/2 or has no limit, according to whether  $\min(p_i,q_i)$ is summable or not.

Let  $N_n:=\mathrm{card}\{i\le n:\ p_i>1/2\};$ then $\lim_{n\to\infty} N_n=\infty.$
Since $$\min(p_i,q_i)=\begin{cases}
p_i.&\text{if }p_i<1/2;\\
q_i=(1-p_i),&\text{if }p_i>1/2,
\end{cases}
$$
we have
\begin{align*}
&\prod_{i=1}^n (1-2p_i)=\prod_{i\le n,p_i\le 1/2}(1-2p_i)\times \prod_{i\le n,p_i>1/2} (1-2p_i)
\\
&=(-1)^{N_{n}}\prod_{i\le n,p_i\le 1/2}(1-2p_i)
\times
\prod_{i\le n,p_i>1/2} (1-2q_i)=(-1)^{N_{n}}
\prod_{i=1}^n (1-2\min(p_i,q_i)).
\end{align*}
Given that $\lim_n (-1)^{N_n}$ does not exist, there are two cases: 

(i)  the right-hand side converges because the product (without the $(-1)^{N_n}$ factor) converges to zero and  mixing holds  ($\sum_{i=1}^\infty\min(p_i,q_i)=\infty)$, in which case $\lim_n y_n=0$ and $\lim_n x_n=1/2$.

(ii) the right-hand side has no limit  and mixing does not hold 
in which case $y_n$ (hence $x_n$) has no limit.

\medskip

\noindent\underline{Case 2: $N<\infty$}. Let us assume first that $N=0$, that is, $p_i<1/2,i\ge 1$. If $\sum_{i=1}^\infty p_i=\infty$ then in \eqref{second.rec} we have $\prod_{i=1}^n (1-2p_i)\downarrow 0$, implying $\lim_n y_n=0$ and $\lim_n x_n=1/2$.  If $\sum_{i=1}^\infty p_i<\infty$, then $\rho=\prod_{i=1}^{\infty} (1-2p_i)>0$ and $\lim_n y_n=\frac12 \rho$, that is, $\lim_n x_n=\frac12(1+\rho).$

In the general case, for large $i$,  $\min(p_i,q_i)=p_i<1/2$, and mixing
is  tantamount to  $\sum_{i=1}^\infty p_i=\infty$. The proof is very similar as before, using the fact that the product has positive terms for large enough indices.

\subsection{Proof of Theorem \ref{warmup}}

The martingale method is applicable in this case too.  Indeed,  direct computation gives $a_n=\frac{1}{2p}$, $\forall n$ and $ v_n=\overset{n}{\sum}4a_i^2p_iq_i=\frac{1-p}{p}n.$ Hence $a_n =o(\sqrt{v_n})$, $a^2_i\xi^2_i$ are bounded, so~\eqref{Drogin.first} holds, and thus we can apply Proposition~D1972, yielding the answer to (INV.M)  in the affirmative.$\qed$

\subsection{Proof of Theorem~\ref{invar}}
First we will prove the statement under the more restrictive assumption that 
\begin{align}\label{qn_behave}
\limsup_{n\to\infty} \frac{q_{n+1}}{q_n}<\infty, 
\end{align}
and then we upgrade it for showing the statement under the condition appearing in the theorem.

\subsubsection{STEP 1} 
We start with a simple lemma.
\begin{lemma}\label{simplemma}
Assume that for the non-negative sequence $(q_n)$, such that
\begin{itemize}
\item $q_n\to 0$,
\item  $\sum_n q_n=\infty$,
\item the condition \eqref{qn_behave} holds.
\end{itemize}
Then $\liminf_{n\to\infty} a_n>0$, where the $a_n$ are defined by~\eqref{an_defined}.
Moreover, if  $\ds\lim_{n\to\infty} \frac{q_{n+1}}{q_n} =1,$ then $\ds \lim_{n\to\infty} a_n=1/2$.
\end{lemma}
\begin{remark}
The condition that  $q_n\rightarrow 0$ is really necessary in Lemma~\ref{simplemma}. Indeed, fix $c_1,c_2>0$, $c_1\neq c_2$ and let $q_n=c_1/n$ if $n$ is odd and $q_n=c_2/n$ if $n$ is even. Then $q_{n+1}/q_n\not\to 1$, though~\eqref{qn_behave} still holds. In this case $a_n\not\to 1/2$, rather (as it is not hard to show) $\displaystyle \lim_{k\to\infty} a_{2k}=\frac{c_1}{c_1+c_2}\ne \frac{c_2}{c_1+c_2}=\lim_{k\to\infty} a_{2k+1}$. $\hfill\diamond$
\end{remark}

\begin{proof}
Fix some $n$, and for $m\ge n$ let
$$
w_m=\prod_{j=n+1}^m (1-2q_{j}), \ m>n, \quad
 w_n=1,
$$ 
and note that $w_m\downarrow 0$ as $m\to\infty$ due to $\sum q_i=\infty$. Then
$$
a_n=\sum_{i=0}^{\infty} (-1)^i w_{n+i}
=\sum_{k=0}^\infty \left(w_{n+2k}-w_{n+2k+1}\right)
=\sum_{k=0}^\infty 2q_{n+2k+1}\,w_{n+2k}.
$$
Now take any finite $c>\limsup_n q_{n+1}/q_n$, and assume that $n$ is so large that the quantity $q_{\ell+1}/q_\ell<c$ for all $\ell\ge n$. Then
\begin{align}\label{eqwn}
w_{n+2k}-w_{n+2k+2}&=  2( q_{n+2k+1} + q_{n+2k+2} -2 q_{n+2k+1} q_{n+2k+2} )\cdot w_{n+2k}\\
&\le 2  ( q_{n+2k+1} + q_{n+2k+2} )\cdot w_{n+2k}
\le  (1+c)\cdot 2q_{n+2k+1} w_{n+2k}.\nonumber
\end{align}
As a result,
$$
(1+c) a_n =\sum_{k=0}^\infty (1+c)\cdot 2q_{n+2k+1}\,w_{n+2k}  \ge 
\sum_{k=0}^\infty \left(w_{n+2k}-w_{n+2k+2}\right)=    w_n=1 >0,
$$
where the telescopic sum converges due to the fact that $w_m\to 0$.
Since $c>\limsup_n q_{n+1}/q_n$ is arbitrary, we can even conclude that
\begin{align}\label{anliminf}
\liminf_{n\to\infty} a_n\ge \frac{1}{1+\limsup_{n\to\infty} q_{n+1}/q_n}.
\end{align}

To prove the second part of the claim, observe that from~\eqref{anliminf} we already have $\liminf_n a_n\ge 1/2$. To show the counterpart of this, fix an arbitrary $\eps>0$ and let  $n$ be so large that 
$ \frac{q_{n+2k+2}}{q_{n+2k+1}}\ge 1-\eps/2$ and 
$q_{n+2k+2}\le \eps/4$ for all $k\ge 0$. Then
\begin{align*}
&q_{n+2k+1} + q_{n+2k+2} -2 q_{n+2k+1} q_{n+2k+2} 
=  \left(
1+\frac{q_{n+2k+2}}{q_{n+2k+1}}-2q_{n+2k+2}\right)
q_{n+2k+1} 
\\
&\ge (2-\eps)\cdot q_{n+2k+1},
\end{align*}
given that $\ds\frac{q_{n+2k+2}}{q_{n+2k+1}}\to 1$ and $q_{n+2k+2}\to 0$. Hence
(see~\eqref{eqwn})
$$
w_{n+2k}-w_{n+2k+2}\ge  (2-\eps)\cdot  2 q_{n+2k+1}w_{n+2k}
$$
and 
$$
(2-\eps) a_n=\sum_{k=0}^\infty  (2-\eps) \cdot 2q_{n+2k+1}\,w_{n+2k} \le 
\sum_{k=0}^\infty \left(w_{n+2k}-w_{n+2k+2}\right)=    w_n=1.
$$
Since $\eps>0$ is arbitrary, we conclude that $\limsup_n a_n\le 1/2$, which completes the proof.
\end{proof}

\medskip
We now continue the proof of the theorem under the assumption that \eqref{qn_behave} holds.

\medskip
Proof of Theorem \ref{invar}~{\sf(a)}: \\
Noting that all conditions at the end of Section \ref{mgale.appr} to Question(M),(INV.M) are in the affirmative (as $a_n$ is well-defined and stays bounded), except \eqref{Drogin.first}. Since in our case $X_i=a_i\xi_i Y_{i-1}$ and $|Y_i|=1$, what we need is to show that
\begin{equation}\label{condition.a}
\lim_{n\to\infty} \frac 1n \sum_{i=1}^{Z(n)} a_i^2 \xi_i^2 \mathbf{1}_{\{a_i^2 \xi_i^2 >n\epsilon\}}=0.
\end{equation}
(Note that $Z(n)$ in our case it is deterministic, and so is $v_m$.) Since
$$
\xi_i^2=[(-1)^{W_{i}}+(2p_i-1)]^2\le 4, \text{ and }  |a_i|\le 1,
$$
as $a_i$ is a Leibniz series, all but finitely many terms in the sum in~\eqref{condition.a} are zero, proving~\eqref{condition.a}. We conclude that \eqref{Drogin.first} holds.

Next, a direct computation shows that $v_m=4\sum_{i=1}^m a_i^2 p_i q_i.$ Then 
$$
\lim_{m\to\infty} v_m=\sum_{i=1}^{\infty} 4 a_i^2 p_i q_i=\infty
$$ 
follows from Lemma \ref{simplemma} and from the assumptions $p_n\to 1$ and $\sum q_n=\infty$. The proof of {\sf (a)} is thus complete.

\medskip
Proof  of  Theorem \ref{invar} {\sf (b)}: First, we prove that $\Lambda_n^2=\sum_{i=1}^n a_i^2 \xi_i^2\to \infty$. Recall that  $\xi_i=(-1)^{W_i}-E(-1)^{W_i}=2p_i-1+(-1)^{W_i}$ satisfies $\E\xi_i^2=\Var((-1)^{W_i})=4p_i q_i$. Let also $U_i:=a_i^2 \xi_i^2\in [0,4]$.

Since the $W_i$ are independent, so are the $\xi_i$, and hence, for $\Lambda_n^2$, the Three Series Theorem applies:  the non-negative series $\sum_i U_i$ diverges if for some  $A>0$, $\sum_i \E[U_i;|U_i|\le A]$ diverges. But for $A>4$, 
$$
\sum_i \E[U_i;\ |U_i|\le A]=\sum_i \E(U_i)=\sum_i a_i^2 p_i q_i=\infty,
$$ 
as $a_i$ is bounded away from zero, $p_i\to 1$ and $\sum q_i=\infty.$ 

Alternatively, let $\epsilon>0$. Then $p_i\to 1$ and $\sum q_i=\infty$ along with the second Borel-Cantelli lemma guarantee that $\xi_i=2p_i-1+(-1)^{W_i}\ge 2-\epsilon$ for infinitely many $i$'s almost surely. We are done because the $a_i$ are bounded away from zero.

For the second statement,  by using  Chebyshev's inequality, it is enough to show that
\begin{align}\label{eq_Cheb}
\lim_{n\to\infty}\frac{\Var(\Lambda_n^2)}{(\E\Lambda_n^2)^2}=0.
\end{align}
Since $a_n,p_n,q_n\in[0,1]$,
\begin{align}\label{eq_Cheb1}
\Var(\Lambda^2_n)=4\sum_{i=1}^n a_i^4p_i q_i(p_i-q_i)^2\le  4\sum_{i=1}^n q_i.
\end{align}
Moreover, for large $n$'s, 
\begin{align}\label{eq_Cheb2}
\E\Lambda_n^2=v_n=4\sum_{i=1}^n a_i^2\, p_i \, q_i
\ge c \sum_{i=1}^n q_i
\end{align}
for some $c>0$, since $\ds \liminf_{i\to\infty} a_i>0$ by Lemma~\ref{simplemma} and $p_i\to 1$. Given that $\sum_{i=1}^n  q_i\to\infty$, \eqref{eq_Cheb1} and~\eqref{eq_Cheb2} together yield \eqref{eq_Cheb}, thus completing the proof of the statement.  $\hfill\qed$

\subsubsection{STEP 2}\label{even.odd}
We now upgrade the result obtained in STEP 1, by dropping the restriction that \eqref{qn_behave} holds. We need the following

\begin{lemma}[Comparison with ``regular" sequences]\label{cool.lemma}
Let $0\le q_n\le 1/2,\ n\ge 1$.

(i) Assume  that there exists a sequence $q^*_k\to 0$ such that $q^*_n$ is not summable, {\em regular}, in the sense that \eqref{qn_behave} holds,  and $q_n\le q^*_n$ for even $n$, while $q_n\ge q^*_n$ for odd~$n$. Then  $\ds\liminf_{k\to\infty} a_{2k}>0$.

(ii) Assume that there exists a sequence $\tilde q_k\to0$  such that $\tilde q_n$ is not summable,  {\em regular} in the sense that \eqref{qn_behave} holds, and $q_n\le \tilde q_n$ for odd $n$, while $q_n\ge \tilde q_n$ for even~$n$. Then  $\ds\liminf_{k\to\infty} a_{2k+1}>0$.
\end{lemma}

\begin{proof}[Proof of Lemma \ref{cool.lemma}]
Since $0\le q_n\le 1/2$ for $n\ge 1$, it is easy to check the following (for example by observing that for $k>n$, the coefficients of $q_k$ in $a_n$ form a Leibniz series as well):
\begin{itemize}
\item Let $n=2k$. Then $a_n$ is decreasing\footnote{The terms {\em increasing} and {\em decreasing} are not used in the strict sense.} in all $q_i$ for which $i$ is even and  increasing in all $q_i$ for which $i$ is odd. 

\item Let $n=2k+1$. Then $a_n$ is increasing in all $q_i$ for which $i$ is even and  decreasing in all $q_i$ for which $i$ is odd. 
\end{itemize}

Turning to the proof of (i) (a similar proof works for (ii), which we omit), note that, because of its monotonicity and non-summability (use $I=\infty$ and $q^*_{2k}\ge q_{2k}$), STEP 1 yields that $(q^*_n)$ is such that $\liminf a_n>0$, and in particular, $\liminf_k a_{2k}>0$. Hence, by the first bullet point above, $\liminf_k a_{2k}>0$ also for $(q_n)$, proving (i).
\end{proof}
\begin{proof}[Proof of Theorem~\ref{invar}] First, without the loss of generality, we assume that $m_0=1$ (changing a finite number of terms does not change the validity of the invariance principle). Similarly, we may and will  assume that $q_n\le 1/2$ for all $n\ge 1$, as we assume anyway that $q_n\to 0$.

We only need that $v_n=4\sum_{i=1}^n a_i^2\, p_i\, q_i\to\infty,$ what is left is very similar to STEP 1. This will follow from $p_n\to 1$ and Assumption~\ref{assA}, provided that either $\liminf_k a_{2k}>0$ or $\liminf_k a_{2k+1}>0$. By Lemma \ref{cool.lemma}, it is sufficient to construct either a sequence~$(q_n^*)$ or a sequence $(\tilde q_n)$ satisfying the properties in the lemma. These sequences will be automatically divergent, given Assumption~\ref{assA} and that~$(q^*_n)$ resp.~$(\tilde q_n)$ dominate~$(q_n)$ for  even resp.\ odd~$n$'s. Now, assume for example \eqref{evens.dominate} (assuming \eqref{odds.dominate} leads to a similar argument). Define
\begin{eqnarray*}
\tilde q_{2m}:=C\max\{q_{2\ell+1},\ell\ge m\},\  m\ge 1;\\
\tilde q_{2m+1}:=\max\{q_{2\ell+1},\ell\ge m\},\  m\ge 0.
\end{eqnarray*}
Then $(\tilde q_n)$ is regular because $\frac{\tilde q_{n+1}}{\tilde q_n}\le \max\{C^{-1},C\}$ for all $n\ge 1$, and trivially $q_{2m}\ge \tilde q_{2m}$ and $\tilde q_{2m+1}\ge q_{2m+1}$. Hence, $\liminf_k a_{2k+1}>0$ by Lemma \ref{cool.lemma}(ii).
\end{proof}
\begin{remark}[One of the two subsequences can be arbitrary]
Chose an arbitrary ``odd" subsequence, satisfying the conditions that it tends to zero and yet not summable. Then take a sufficiently large ``even" subsequence that dominates it in the sense of~\eqref{evens.dominate}, but still tends to zero (for example, let  $q_{2n}:=1/\sqrt{2n}$ and $q_{2n+1}:=1/(2n+1)$). Then \eqref{evens.dominate} holds, while the condition $\limsup_n q_{n+1}/q_n<\infty$ (cf. \eqref{qn_behave} in the proof) fails to hold, as $\lim_n q_{2n}/q_{2n+1}=\infty$.

By the same token, one can first chose an arbitrary non-summable ``even" sequence, with the terms tending to zero and then a dominating ``odd" one.$\hfill\diamond$
\end{remark}

\subsection{Proof of Proposition~\ref{equiv.for.b}}
Recall that $S_n=M_n+(1-a_n)Y_n$, hence
$$
\Var(S_n)=\Var(M_n)+(1-a_n)^2+2(1-a_n)\Cov(M_n, Y_n),
$$ 
where, by Cauchy-Schwarz, $|\Cov(M_n, Y_n)|\leq \sqrt{\E(M_n^2)}= \sqrt{v_n}$,  so 
\begin{align*}
\sigma_n^2-v_n=(1-a_n)(1-a_n+2\Cov(M_n, Y_n))=(1-a_n)(1-a_n+A_n\sqrt{v_n}),
\end{align*} 
where $|A_n|\le 1$. Then
$$
\frac{\sigma^2_n}{v_{n}}-1=\frac{1-a_n}{\sqrt{v_n}}\cdot \left(\frac{1-a_n}{\sqrt{v_n}}+A_{n}\right),
$$
if $v_n\to \infty$ and $a_n = o(\sqrt{v_n})$ as $n\to\infty$, hence  $\sqrt{v_n}\sim \sigma_n$ follows. 

Similarly, we have
\begin{equation*}
1-\frac{v_{n}}{\sigma_n^2}=\frac{1-a_n}{\sigma_{n}}\left(\frac{1-a_n}{\sigma_{n}}+A_n\sqrt{v_{n}}/\sigma_n\right).
\end{equation*} 
Using the shorthands $\ds w_n:=\frac{\sqrt{v_{n}}}\sigma_n$ and $\ds b_n:=\frac{1-a_{n}}{\sigma_{n}}$, one obtains  the quadratic equation $w_n^2+b_nA_nw_n+b_n^2-1=0,$ where $b_n\to 0$. Hence
$$
w_n=\frac{-b_nA_n\pm\sqrt{b_n^2A_n^2+4(1-b^2_n)}}2,
$$
but of course $w_n\ge 0$. Therefore, $b_n\to 0$ implies that $w_n\to 1$, that is, $\sqrt{v_n}\sim \sigma_n$. This is clearly the case when  $\sigma_n\to\infty$   and $a_n =o(\sigma_n)$ as $n\to\infty.\hfill\qed$

\subsection{Proof of Theorem \ref{Lower.Main.Thm} -- strongly critical case}
First, it is easy to see that if~$X$ is a symmetric random variable, concentrated on $[-t,t]$, then $\Var(X)\le t^2$, with equality if and only if the law of~$X$ is $\frac{1}{2}(\delta_{-t}+\delta_{t})$. 

Now assume that $\lim_{n\to\infty}np_n=0$. 
By a well-known criterion for tightness (see Theorem 4.10 in \cite{KS1991}), the laws of the $S^{(n)}$ are tight on $C([0,T])$ if besides the condition $\lim_{\eta\to+\infty} \sup_{n\geq 1} \P(S^{(n)}(0)>\eta)=0$, one also has
\begin{align*} 
\lim_{\delta\downarrow 0} \sup_{n\geq 1} \P\left(\underset{0\leq t,s\leq t_k}{\max_{|t-s|\leq \delta}}|S^{(n)}(t)-S^{(n)}(s)|>\epsilon\right)=\lim_{\delta\downarrow 0}\P(\delta>\epsilon)=0,\ \forall \epsilon>0.
\end{align*}
Since $S^{(n)}(0)=0,\ n\ge 1$, the first condition clearly holds. The second one is satisfied by the uniform Lipschitz-ness:  $|S^{(n)}(t)-S^{(n)}(s)|\leq |t-s|,\ n\ge 1$.

Given tightness on $C([0,T])$, it is sufficient to show that the limit at time $t>0$ is $\frac{1}{2}(\delta_{-t}+\delta_{t})$, that is, it satisfies $\Var(X)\ge t^2$. Indeed, the only continuous functions $f$ on $(0,T)$ satisfying  $|f(t)|=t$ are $f(t)=t$ and $f(t)=-t$. For simplicity we will work with $t=1$ (otherwise use a simple scaling), that is we will show that  every partial limit at time $t=1$ is such that its variance is at least one.

To achieve this, fix $N\ge 1$ and recall from~\cite{EV2018} (see the two displayed formulae right before Theorem 3 there) that 
$$
\Var\left(\frac{S_N}N \right)=\frac 1N +\frac{2}{N^2}\sum_{1\le i_1<i_2\le N}e_{i_{1},i_2}.
$$
This quantity is monotone decreasing in all $p_n$'s as long as they are all less or equal than $1/2$, because the same holds for each fixed $e_{i,j}$. Fix $\epsilon>0$ and let $N=N(\epsilon)$ be such that $\epsilon/N\le 1/2$ and that also $\epsilon/n>p_n$ holds for all $n>N$. Define $\hat p_n$ so that it coincides with $p_n$ for $n\le N$ and $\hat p_n=a/n$ for $n>N$. By  monotonicity, 
$$
\Var\left(\frac{S_{n}}{n}\right)\ge \Var\left(\frac{\hat S_{n}}{n}\right),\ n\ge 1,
$$
where $\hat S$ is the walk for the sequence $(\hat p_n)$.  

In~\cite{EV2018} it was shown that 
$$
\lim_{n\to\infty}\Var\left(\frac{\hat S_{n}}{n}\right)=\frac{1}{2\epsilon+1}
\quad\Longrightarrow\quad
\liminf_{n\to\infty}\Var\left(\frac{S_{n}}{n}\right)\ge\frac{1}{2\epsilon+1}.$$
Since $\epsilon>0$ was arbitrary, 
$$
\liminf_{n\to\infty}\Var\left(\frac{S_{n}}{n}\right)\ge 1.
$$
Now, if $S_{n_{j}}/n_j\to X$ in law, then 
$$
\lim_{j\to\infty}\Var\left(\frac{S_{n_{j}}}{n_{j}}\right)=\Var(X),
$$ 
because $\E(S_n)=0$ and the variables are all supported in $[-1,1]$ (and so the test function $f(x)=x^2$ is admissible). From the last two displayed formula, we have that $\Var(X)\ge 1$ and we are done.
$\qed$

\subsection{Proof of Theorem \ref{Lower.Main.Thm} -- supercritical case}
By the Borel-Cantelli Lemma, for almost every $\omega$, either $S_n(\omega)=1$ for all large $n$ or $S_n(\omega)=-1$ for all large $n$. As $n\to\infty$, in the first case the path converges uniformly to a straight line with slope~$1$; in the second case it converges uniformly to a straight line with slope $-1$. $\qed$


\subsection{Proof of Theorem \ref{Lower.Main.Thm} -- critical case}
Fix $T>0$, denote by ${\Mc}_T$ the set of all locally finite point measures on the interval $(0, T]$, and denote by $N^{(n)}=N^{(n,T)}$ the laws of the point processes induced by the turns of the walk $S^{(n)}$ on the time interval $(0,T]$. 

Let $t\in (0,T)$; we {\bf assign a continuous (zigzagged) path that increases at\footnote{I.e. it increases on $[t,t+\epsilon]$ for some small $\epsilon>0$.}~$t$ to each point measure.}
\begin{definition}[Assigning paths]
Define the map $\Phi_{t}: {\Mc}_T \rightarrow C[0,T]$ as follows.
\begin{itemize}
\item First, label the (countably many) atoms on $(0, t]$ from right to left as $a_1,a_2,...,$ i.e., the closest one on the left to $t$ as $a_1$, the second closest as $a_2$, etc., and note that $t=a_1$ is possible; also label the atoms on $(t, T]$,  from the closest to the furthest as $b_1, b_2$,...;
\item assign ``$+$" sign to the  intervals (the union of which is denoted by $S^+_t$) $$...[a_7, a_6),  [a_5, a_4), [a_3, a_2), [a_1, b_1), [b_2,b_3), [b_4,b_5), [b_6,b_7),...;$$
\item assign ``$-$" sign to the intervals (the union of which is denoted by $S^-_t$) $$...[a_8, a_7) , [a_6, a_5), [a_4, a_3), [a_2, a_1), [b_1,b_2), [b_3,b_4), [b_5,b_6),...$$
 \end{itemize}
Let  $\mu\in {\Mc}_T$. For  $0<r\leq T$,  define
\begin{align}\label{Phi} 
\Phi_t(\mu)(r):= L((0,r]\cap S^+_t)-L((0,r]\cap S^-_t),\ \mathrm{with}\ \Phi_t(\mu)(0):=0,
\end{align}
where $L$ is the Lebesgue measure on the real line. Then $\Phi_t(\mu)(\cdot)$ is well-defined and continuous on $[0,T]$. Intuitively,  it  describes the difference between the total length of increasing parts and the total length of  decreasing parts, assuming increase at $t$. Clearly,
\begin{align} \label{Phi_r}
|\Phi_t(\mu)(r)|\leq r,\  0< r\le T.
\end{align}
\end{definition}
\begin{remark}\label{formerlynote}
The case $t=0$ is excluded, i.e. one cannot set the  path $\Phi_t(\mu)(\cdot)$ to first increase at $t=0$, as our point measures  may not be locally finite around $0$. For instance, we will show that $N_n$ converges to a limiting Poisson Point Process (PPP)~$N$, and this $N$ explodes at $0$. However, for $t>0$, $\Phi_t(r)\rightarrow 0$, as $r\rightarrow 0$.
\end{remark}

We now turn to the case of a PPP with intensity $\frac{c}{x}$ (we replaced the constant $a$ of Theorem \ref{Lower.Main.Thm} by $c$ in the proof to avoid confusion).
\begin{proposition}\label{turn_distribution}(Turning points $\to$ PPP with intensity  $\frac{c}{x}$)
Given  $0<a<b<\infty$, $c>0$, set $p_n=\frac{c}{n}\wedge 1,$ and denote the number of turns from step $\lceil an\rceil+1$ to step $\lceil bn\rceil$ by $N^{(n)}((a,b])$. Denoting $\mu_{c;a,b}:=c\ln(b/a)=\int_a^b \frac{c}{x}\, \mathrm{d}x,$ one has
\begin{itemize}
\item[(i)] for $k\ge 0,\ 0<a<b$, as $n\to\infty$,
\begin{align}\label{estimation}
\P\left( N^{(n)}((a,b])=k \right)&=\exp(-\mu_{c;a,b}) \frac{\mu_{c;a,b}^k}{k!}+O\left(\frac{1}{n}\right);\\
 {\sf Law}(N^{(n)}((a,b]))&\overset{n\rightarrow\infty}{\longrightarrow} {\sf Poiss}(\mu_{c;a,b});\label{limit_law}
\end{align}
\item [(ii)] given $0< t_1<t_2<...<t_l<\infty$, the random variables $$N^{(n)}( (t_1,t_2]),N^{(n)}((t_2,t_3]),...,N^{(n)}((t_{l-1},t_l])$$ are independent (independent increments), and
\begin{align*}
&{\sf Law}\left( N^{(n)}((t_1,t_2]), N^{(n)}((t_2,t_3]),...,N^{(n)}((t_{l-1},t_l] )\right) \\
&\ \ \ \overset{n\longrightarrow\infty}{\longrightarrow}{\sf Poiss(c)} \left( (\mu_{c;t_1,t_2}), (\mu_{c;t_2,t_3})...,(\mu_{c;t_{l-1},t_l}) \right),
\end{align*}
where ${\sf Poiss}(c)={\sf Poiss}((0,\infty),\frac{c}{x}\, \mathrm{d}x)$ is the law of the PPP with intensity $\frac{c}{x}\, \mathrm{d}x$ on $(0,\infty)$.
\end{itemize}
\end{proposition}

\begin{proof}(of Proposition \ref{turn_distribution}:)

\noindent{\sf STRATEGY OF THE PROOF:}\
We first prove part (i). Once that is done,  since the turns from step $\lceil t_in\rceil +1$ to step $\lceil t_j n\rceil$ and from $\lceil t_l n\rceil+1$ to  $\lceil t_j n\rceil$ are independent for any $0<t_i<t_j\leq t_l<t_r<\infty$, part (ii) will immediately follow.

Regarding part (i), we only need to prove equation \eqref{estimation}, and then \eqref{limit_law} will easily follow. In fact  we only give here the proof (in three steps) of \eqref{estimation} for $a,b$ integers, i.e., $\lceil an\rceil=an$, $\lceil bn\rceil=bn$, for $n$ large enough; the proof for general $0<a<b$ can then be easily adjusted. 

\medskip
\noindent {\sf STEP 1:}\ Given  $c>0$, and  $n$ large enough,  define 
$$\Pi_{c,n}:= \P(\text{no turn between}\ an+1 \ \text{and}\ bn).$$ 
 We now provide an estimate for $\Pi_{c,n}$, namely
\begin{align}\label{estimation_Pi}  
\Pi_{c,n}=\exp(-\mu_{c; a,b})+O\left(\frac{1}{n}\right).
\end{align}
Indeed,
\begin{align*}
\Pi_{c,n}=& \frac{an+(1-c)}{an+1}\cdot\frac{an+1+(1-c)}{an+2}\cdot\frac{an+2+(1-c)}{an+3}\cdot...\cdot\frac{bn-1+(1-c)}{bn}\\
=&\left( \frac{an+(1-c)}{an}\cdot \frac{an+1+(1-c)}{an+1}...\frac{bn-1+(1-c)}{bn-1} \right)
\\
&\times 
\left( \frac{an}{an+1}\cdot\frac{an+1}{an+2}\cdot...\cdot\frac{bn-1}{bn} \right)
\\
=&\frac{a}{b}\cdot \left( \frac{an+(1-c)}{an}\cdot \frac{an+1+(1-c)}{an+1}...\frac{bn-1+(1-c)}{bn-1} \right)\\
=&\frac{a}{b}\exp\left( \sum_{i=1}^{bn-an}
\ln(an+i-c)-\ln(an+i-1)
\right)
=\frac{a}{b}\exp\left(\sum_{i=1}^{bn-an}\displaystyle\int_{an+i-1}^{an+i-c} \frac{\mathrm{d}x}{x} \right).
\end{align*}
The exponent tends to $(1-c)\ln\frac ba$, and so $\lim_{n\rightarrow\infty}\Pi_{c,n}=\exp(-c\ln(b/a))=\exp(-\mu_{c; a,b})$. Indeed, 
\begin{align*}
\frac{1-c}{an+i-c}\leq \int_{an+i-1}^{an+i-c} \frac{1}{x}\,\mathrm{d}x \leq \frac{1-c}{an+i-1},
\end{align*}
hence
\begin{align*}
\sum_{i=1}^{bn-an}\frac{1-c}{an+i-c}\leq\overset{bn-an}{\sum_{i=1}} \int_{an+i-1}^{an+i-c} \frac{1}{x}\,\mathrm{d}x \leq \overset{bn-an}{\sum_{i=1}}\frac{1-c}{an+i-1},
\end{align*}
where $\ds\lim_{n\rightarrow\infty}\overset{bn-an}{\sum_{i=1}}\frac{1}{an+i-c}=\lim_{n\rightarrow\infty}\overset{bn-an}{\sum_{i=1}}\frac{1}{an+i-1}=\ln(b/a)$. leading to~\eqref{estimation_Pi}.

\medskip
\noindent {\sf STEP 2:}  we now estimate $$\P(N^{(n)}((a,b])=1)=\P(\text{there is one turn from step}\ an+1 \text{to step}\ bn).$$ Note that the turning step can happen at step $an+i$, for $i=1,2,...,bn-an$, with corresponding probabilities  $(\frac{an+1-c}{an+1}\cdot \frac{an+2-c}{an+2}\cdot...\cdot\frac{bn-c}{bn})\cdot \frac{c}{an+i}=\Pi_{c,n}\cdot \frac{c}{an+i},$ $i=0,1,...,bn-an-1$. Thus, 
\begin{align*}
\P(N^{(n)}((a,b])=1)=\Pi_{c,n}\overset{bn-an-1}{\sum_{i=0}} \frac{c}{an+i}=\Pi_{c,n}\cdot c\cdot\Delta_n,
\end{align*}
where 
\begin{align*}
\Delta_n = \overset{bn-an-1}{\sum_{i=0}} \frac{1}{an+i}.
\end{align*}
Since
\begin{align*}
\ln\frac ba = \int_{an}^{bn}\frac{\mathrm{d}x}{x}\leq \Delta_n\leq \int_{an}^{bn}\frac{\mathrm{d}x}{x}+\left(\frac{1}{an}-\frac{1}{an+1}\right)(bn-an)=\ln\frac ba+\frac{(b-a)}{a(an+1)},
\end{align*}
one has
\begin{align}\label{estimation_Lambda}
\Delta_n=\ln\frac ba+O(1/n), 
\end{align}
and then \eqref{estimation_Pi}, \eqref{estimation_Lambda} give
\begin{align*}
\P\left( N^{(n)}((a,b])=1 \right) =\frac{\mu_{c; a,b}}{1!}e^{-\mu_{c;a,b}}+O\left(\frac{1}{n}\right).
\end{align*}

\medskip
\noindent {\sf STEP 3:} we verify \eqref{estimation} using induction, and so we assume that
\begin{align}\label{induction_k}
\P\left( N^{(n)}((a,b])=k \right)=\exp(-\mu_{c;a,b})\frac{\mu_{c;a,b}^k}{k!}+O\left(\frac{1}{n}\right),
\end{align}
and show that $k$ can be replaced by $k+1$ as well. On the the event $\{N^{(n)}((a,b])=k\}$, there should be $k$ turns from step $an+1$ to step $bn+1$, say the turns happen at $an+i_1, an+i_2,...,an+i_k$, where $i_1,...,i_k$ is an increasing sequence taking values in $\{0,1,...,bn-an-1\}$. Similarly to the $k=1$ case, the probability for this to happen is  
\begin{align*}
p=\Pi_{c,n}\cdot \left(\frac{c}{an+i_1}\frac{c}{an+i_2}...\frac{c}{an+i_k}\right).
\end{align*}
Then  $\P(N^{(n)}((a,b])=k)$ is the sum of all such terms, i.e.,
\begin{align*}
 \P(N^{(n)}((a,b])=k)=\Pi_{c,n}\cdot c^k\cdot\sum_{0\leq i_1<\dots<i_k\leq bn-an-1}\frac{1}{an+i_1}\frac{1}{an+i_2}...\frac{1}{an+i_k}.
\end{align*} 
By assumption \eqref{induction_k} and the estimate~\eqref{estimation_Pi}, we have
\begin{align}\label{induction} 
 \sum_{0\leq i_1<\dots<i_k\leq bn-an-1}\frac{1}{an+i_1}\frac{1}{an+i_2}...\frac{1}{an+i_k}=\frac{(c \ln(\frac ba))^k}{k!}+O\left(\frac{1}{n}\right)=\frac{\mu_{c;a,b}^k}{k!}+O\left(\frac{1}{n}\right).
\end{align}
Similarly,
\begin{align*}
\P(N^{(n)}((a,b])=k+1)=\Pi_{c,n}\sum_{0\leq i_1<\dots<i_{k+1}\leq bn-an-1}\frac{c}{an+i_1}\frac{c}{an+i_2}\dots\frac{c}{an+i_{k+1}},
\end{align*}
where the sequence $i_1<i_2<...<i_k<i_{k+1}$ takes values in $\{0,1,...,bn-an-1\}$. Now 
\begin{align*}
\frac{c}{an+j}\P(N^{(n)}((a,b])=k)=\Pi_{c, n}\sum_{0\leq i_1<i_2<...<i_k\leq bn-an-1}\frac{c}{an+i_1}\frac{c}{an+i_2}...\frac{c}{an+i_k}\frac{c}{an+j},
\end{align*}
for $j=0,1,...,bn-an-1$. Now consider the sum
\begin{align}\label{main_step}
&\left(\sum_{j=0}^{bn-an-1}\frac{c}{an+j}\right)  \P\left( N^{(n)}((a,b])=k \right) \\
=& \overset{bn-an-1}{\sum_{j=0}} \left( \Pi_{c,n}\cdot\sum_{0\leq i_1<\dots<i_k\leq bn-an-1}\frac{c}{an+i_1}\frac{c}{an+i_2}...\frac{c}{an+i_k}\frac{c}{an+j}  \right).\nonumber
\end{align} 
In each sum on the right-hand side, there are two different kinds of terms: terms of the type 
$$
\frac{c}{an+i_1}\frac{c}{an+i_2}...\frac{c}{an+i_k}\frac{c}{an+i_{k+1}},
$$
where $i_m, m=1,2,...,k+1$ are all different (no repetitions), and terms of the type
$$
\frac{c}{an+i_1}\frac{c}{an+i_1}\frac{c}{an+i_2}...\frac{c}{an+i_k},
$$
where $i_m, m=1,2,...,k$ are all different (one repetition). We then rearrange the right-hand side: sum the ``non-repeating" terms as one group, denoted by $I$; sum the  ``once repeating" ones where the term $\frac{c}{an+j}$ is the one repeated by $I_j$, $j=0,1,..., bn-an-1$, and we estimate $I$, $I_j$ separately.
\begin{align*} 
I=&(k+1)\cdot \left( \Pi_{c,n}\sum_{i_1<i_2<...<i_k<i_{k+1}}\frac{c}{an+i_1}\frac{c}{an+i_2}...\frac{c}{an+i_k}\frac{c}{an+i_{k+1}} \right)\\
=&(k+1)\cdot \P\left( N^{(n)}((a,b])=k+1 \right),
\end{align*}
since each product $\frac{c}{an+i_1}\frac{c}{an+i_2}...\frac{c}{an+i_k}\frac{c}{an+i_{k+1}}$ appears $k+1$ times in sum $I$. Further,
\begin{align*}
 I_j=&\frac{c^2}{(an+j)^2}\left(  \Pi_{c,n}\underset{i_m\neq j}{\sum_{0\leq i_1<i_2<...<i_k\leq bn-an-1}}\frac{c}{an+i_1}\frac{c}{an+i_2}...\frac{c}{an+i_k} \right)\\
 \leq& I_0
 =\frac{c^2}{(an)^2}\left(  \Pi_{c,n}\sum_{1\leq i_1<i_2<...<i_k\leq bn-an-1}\frac{c}{an+i_1}\frac{c}{an+i_2}...\frac{c}{an+i_k} \right)\\
 \leq &\frac{c^2}{(an)^2}P\left( N^{(n)}((a,b])=k \right)=\frac{c^2}{(an)^2}\left( \exp(-\mu_{c;a,b})\frac{\mu_{c;a,b}^k}{k!}+O\left(\frac{1}{n}\right) \right),
\end{align*}
hence
\begin{align*}
\sum_{j=0}^{bn-an-1} I_j &\leq (bn-an)\cdot (I_0)\leq \frac{bn-an}{(an)^2}\cdot \left( \frac{\mu_{c;a,b}^k}{k!} e^{-\mu_{c;a,b}}+O\left(\frac{1}{n}\right) \right)
=O\left(\frac1n\right).
\end{align*}
Now, by estimations of $I, I_j$s, \eqref{main_step} is written as
\begin{align*}
&(k+1)\cdot \P\left( N^{(n)}((a,b])=k+1 \right)+O\left(\frac{1}{n}\right)=I+\overset{bn-an-1}{\sum_{j=1}}I_j\\
 =&\left(\sum_{j=0}^{bn-an-1}\frac{c}{an+j}\right)  \P\left( N^{(n)}((a,b])=k \right)\\
 =&\left( \mu_{c;a,b}+O\left(\frac{1}{n}\right) \right)\cdot\left( \exp(-\mu_{c;a,b})\frac{\mu_{c;a,b}^k}{k!}+O\left(\frac1n\right) \right)
 =\frac{\mu_{c;a,b}^{k+1}}{k!} e^{-\mu_{c;a,b}}+O\left(\frac1n\right),
\end{align*}
and we conclude that \eqref{induction_k} holds with $k$ replaced by $k+1$. This completes the proof of Step 3, and of the proposition altogether.
\end{proof}
{\bf Note:} We use the endpoints $\lceil an \rceil +1, \lceil bn \rceil$ because $\frac{\lceil an\rceil+1,}{n}\rightarrow a^+, \frac{\lceil bn\rceil}{n}\rightarrow b$, so the above limit represents the number of turns in  (a,b] in the scaling limit.
\\[5mm]
In the sequel we will consider measures equipped with both the weak and the vague topologies.  When we consider laws on $C([0,T],\|.\|_{[0,T]})$ where $\|.\|_{[0,T]}$ denotes supremum norm, weak convergence is denoted by $\overset{w}{\rightarrow}$. When one uses vague topology for measures and random measures are considered, $\mathcal X_n\overset{vd}{\rightarrow}\mathcal X$ will be used for convergence  in distribution.

\begin{proposition}[Convergence for point measures and  paths]\label{N_and_Phi}

Let $0<t<T$. Then
\begin{itemize}
\item[(i)] As $n\to\infty$, $N^{(n)}\overset{vd}{\rightarrow} {\sf Poiss}(c)$ on $\mathcal M_T$ equipped with the vague topology, where ${\sf Poiss}(c)$ is the PPP on $(0, T]$ with intensity $\frac{c}{x}\, \mathrm{d}x$.
\item[(ii)] $\Phi_t:{\Mc}_T\to C[0,T]$ is a continuous and uniformly bounded functional, when the former space is equipped with the vague topology, and the latter with  the supremum norm $\|.\|_{[0,T]}$.
\item[(iii)]  As $n\to\infty$, $\Phi_t(N^{(n)})\overset{w}{\rightarrow}\Phi_t({\sf Poiss}(c))$ on $C([0,T],\|\cdot\|)$.
\end{itemize}
\end{proposition}
\begin{proof}
(of Proposition \ref{N_and_Phi}:)
{\sf (i)} In order to use Lemma \ref{lemma_vague} of the Appendix, one needs to define a new metric on $(0, T]$ by $\rho(x,y):=|1/x-1/y|$. Then $\Delta :=((0,T],\rho)$ is a complete separable metric space; notice that $(0,\epsilon]$ is not bounded under $\rho$. Setting $\Ic:=\{(a,b], 0<a<b\leq T\}$,  it is obvious that $\Ic$ is a semi-ring of bounded Borel sets in $\Delta$, and $\mu(\partial (a,b])=\mu(\{a\}\cup \{b\})=0$, hence $\Ic\subset \widehat{\Delta}_{E {\sf Poiss}(c)}$, where  $\widehat{\Delta}_{E {\sf Poiss}(c)}$ is the class of all bounded sets $A \subset \Delta$ with $E\sf Poiss(c) (\partial A)=0$. Then by Lemma \ref{lemma_vague} of the Appendix, we only need to prove, $N^{(n)}(f)\overset{d}{\rightarrow} {\sf Poiss}(c)(f)$, for any $f\in \hat{\Ic}_+$, i.e., any $f$ with $f=\overset{k}{\sum_{i=1}}c_i1_{(a_i,b_i]}$, where  $(a_i, b_i]\in \Ic$ and $a_i>0$. Note  that $f$ is undefined on $(0, \min a_i]$. Then  $N^{(n)}\overset{vd}{\rightarrow} {\sf Poiss}(c)$ on $(0,T]$ follows from 
$
N^{(n)}(1_{(a,b]})\overset{d}{\rightarrow}{\sf Poiss}(c)(1_{(a,b]})$ for $0<a<b\leq T,$ which in turn, follows from Proposition \ref{turn_distribution}.

{\sf (ii)} Assume that $\mu_n,\mu \in {\Mc}_T$, and $\mu_n\overset{v}{\rightarrow} \mu$. Then for any $\varepsilon>0$ small enough, $\mu_n\overset{v}{\rightarrow} \mu\ \text{on}\ [\varepsilon, T].$ Since $\mu$ is locally finite, it has finitely many atoms on $[\varepsilon, T]$, say $\varepsilon\leq x_1 \leq ...\leq x_l\leq T$. It easy to see that $\exists\ n_0$ such that for any $n\ge n_0$, $\mu_n$ also has~$l$ atoms there. Moreover, $\exists\ K=K(\varepsilon,l)\geq n_0$, such that, for any $n\geq K$,
\begin{align*} 
 |x_i^{(n)}-x_i|\leq \frac{\varepsilon}{2(l+2)^2} , \quad \text{for all}\ i=1,2,...,l.
\end{align*}
By \eqref{Phi_r}, $|\Phi_t(\mu_n)(\varepsilon)-\Phi_t(\mu)(\varepsilon)|\leq 2\varepsilon,$
and by definition \eqref{Phi}, we have
\begin{align*}
|\Phi_t(\mu_n)(t)-\Phi_t(\mu)(t)|\leq \frac{l+2}{(l+2)^2}\varepsilon,\ \text{so}\\
\|\Phi_t(\mu_n)-\Phi_t(\mu)\|_{[0, T]}\leq \frac{(l+2)^2}{(l+2)^2}\varepsilon=\varepsilon,\ n\ge K.
\end{align*}
Hence, $\Phi_t$ is continuous. Moreover, $\|\Phi_t(\mu)\|_{[0,T]}\leq T$, so $\Phi_t$ is also uniformly bounded. 

Finally, (iii) immediately follows from (i), (ii) and Lemma \ref{lemma_vague}, completing the proof of Proposition \ref{N_and_Phi}.
\end{proof}

Having Proposition \ref{N_and_Phi} at our disposal, it is now easy to prove that the processes $S^{(n)}$ in the statement of the theorem converge in law to the zigzag process, by checking the convergence of the finite dimensional distributions, and then tightness. 

\medskip
\noindent\underline{Convergence of fidi's:} Given $0<t_1<t_2<...<t_k$, to check that the law of $(S^{(n)}_{t_1},..., S^{(n)}_{t_k})$ converges as $n\to\infty$, let $A_1,A_2,...,A_k\subset\R$ be Borel sets, and denote 
$$
\vec{A}:=(A_1,...,A_k),\quad 
-\vec{A}:=(-A_1,...,-A_k);
$$
$$
(S^{(n)}_{\vec{t}}\in \vec{A}):=\left( S^{(n)}_{t_1}\in A_1,...,S^{(n)}_{t_k}\in A_k \right);
$$
$$
(\Phi_t(u)_{\vec{t}}\in \vec{A}):=\left( \Phi_t(u)_{t_1}\in A_1,...,\Phi_t(u)_{t_k}\in A_k \right).
$$
Moreover, $\{S^{(n)}_s=+\}$ ($\{S^{(n)}_s=-\}$) will denote the event that the zigzag path is increasing (decreasing) at $s^+$, by which we mean that there exists a small interval $[s,s+\epsilon]$ such that $S^{(n)}$ has slope $1$ ($-1$) on $(s,s+\epsilon)$. Then
\begin{align*}
\P\left( S^{(n)}_{\vec{t}}\in \vec{A} \right)=&
\P\left(S^{(n)}_{\vec{t}}\in \vec{A}\mid S^{(n)}(t_1)=+ \right)\P\left(S^{(n)}(t_1)=+ \right)\\
+&
\P\left(S^{(n)}_{\vec{t}}\in \vec{A}\mid S^{(n)}(t_1)=- \right)\P\left(S^{(n)}(t_1)=- \right),
\end{align*}
where, by symmetry, 
$\P\left(S^{(n)}(t_1)=+ \right)=\P\left( S^{(n)}(t_1)=- \right)=\frac{1}{2},$ and
\begin{align*}
\P\left(S^{(n)}_{\vec{t}}\in \vec{A}\mid S^{(n)}(t_1)=+ \right)
=\P\left( \Phi_{t_1}(N^{(n)})_{\vec{t}}\in \vec{A} \right).
\end{align*}
By Proposition~\ref{N_and_Phi},  $\Phi_{t_1}(N^{(n)})\overset{w}{\to} \Phi_{t_1}({\sf Poiss}(c))$ on $C[0, t_k]$;  composing with projections  yields
\begin{align*}
\P\left(S^{(n)}_{\vec{t}}\in \vec{A}\mid S^{(n)}(t_1)=+ \right)\overset{n\rightarrow \infty}{\longrightarrow} \P\left( \Phi_{t_1}({\sf Poiss}(c))_{\vec{t}}\in \vec{A} \right).
\end{align*}
Similarly,
\begin{align*}
\P\left(S^{(n)}_{\vec{t}}\in \vec{A}\mid S^{(n)}(t_1)=- \right)
=\P\left( -S^{(n)}_{\vec{t}}\in \vec{-A}\mid -S^{(n)}(t_1)=+ \right)
\end{align*}
tends to $\P\left( \Phi_{t_1}({\sf Poiss}(c))_{\vec{t}}\in -\vec{A}\right)$ as $n\rightarrow \infty$, hence
\begin{align*}
\P\left(S^{(n)}_{\vec{t}}\in \vec{A} \right) \overset{n\rightarrow \infty}{\longrightarrow} 
\frac{1}{2}\left( \Phi_{t_1}({\sf Poiss}(c))_{\vec{t}}\in \vec{A}\right)+\frac{1}{2}\left( \Phi_{t_1}({\sf Poiss}(c))_{\vec{t}}\in -\vec{A}\right).
\end{align*}

\medskip
\noindent\underline{Tightness:} We repeat the argument in the proof of Theorem \ref{Lower.Main.Thm} here. Use  that $\lim_{\eta\to+\infty} \sup_{n\geq 1} \P(S^{(n)}(0)>\eta)=0$ together with
\begin{align*} 
\lim_{\delta\downarrow 0} \sup_{n\geq 1} \P\left(\underset{0\leq t,s\leq t_k}{\max_{|t-s|\leq \delta}}|S^{(n)}(t)-S^{(n)}(s)|>\epsilon\right)=\lim_{\delta\downarrow 0}\P(\delta>\epsilon)=0,\ \forall \epsilon>0
\end{align*}
are sufficient for tightness  on $C([0,T])$. These are indeed satisfied because $S^{(n)}(0)=0,\ n\ge 1$ and because of the uniform Lipschitz-ness.
This completes the proof of  the theorem in the critical case. $\hfill\qed$

\smallskip
\noindent{\bf Note:} One can use any $\Phi_s,\, s>0$, instead of $\Phi_{t_1}$ (again, $s=0$ is excluded), without causing too much change;  then
\begin{equation*}
\P\left(X^{(n)}_{\vec{t}}\in \vec{A}\mid X^{(n)}(s)=+ \right)\overset{n\rightarrow \infty}{\longrightarrow} \P\left( \Phi_{s}({\sf Poiss}(c))_{\vec{t}}\in \vec{A} \right).
\end{equation*}

\begin{remark} We can also generalize the condition $A_n:=np_n=c$ a bit, namely, one can mimic the proof in Proposition \ref{turn_distribution} to show the following.\\
If the $A_n$ are stable in the sense that
$$
\sum_{k=an}^{bn} \frac{A_k-c}{k}\overset{n\to \infty}{\longrightarrow} 0,
\quad\text{that is}\quad 
\sum_{k=an}^{bn} \frac{A_k}{k}\overset{n\to \infty}{\longrightarrow}c\ln(b/a),\quad \forall 0<a<b<\infty,
$$
then the turns $N^{(n)}$ tend to a PPP with intensity $\lambda(x)=\frac{c}{x}\,\mathrm{d} x$. Hence the law of $S^{(n)}$ tends to that of the same zigzag process, i.e., we have the same scaling limit. This includes, for example, the following cases:
\begin{itemize}
\item $A_n\equiv c$ for all large $n$; 
\item $\lim\limits_{n\rightarrow \infty}A_n=c$; 
\item $A_n$ is periodic with average  period  $c$,
\end{itemize}
where $c$ is a positive constant.$\hfill\diamond$
\end{remark}

\subsection{Proof of Theorem \ref{Lower.Main.Thm} -- subcritical case}
Following the martingale approximation approach and again to prove all conditions at the end of Section \ref{mgale.appr}, we will prove the result in the following steps:
\begin{itemize}
\item[(i)] The $a_n\ge 1$ are well-defined; furthermore 
$a_n=o(n)$;
\item[(ii)] $\lim_{m\to\infty}v_m= \infty$;
\item[(iii)]  $a^2_n=o(v_n)$;
\item [(iv)] As $n\to\infty$,
\begin{equation}\label{drogin(a)}
\frac{1}{n}\overset{Z(n)}{\sum\limits_{i=1}}a_i^2\xi_i^21_{\{a_i^2\xi_i^2>n\varepsilon\}}\overset{L^1}{\longrightarrow}0.
\end{equation}
\end{itemize}
\medskip

\noindent\underline{Step (i).}
Since $1-x\le e^{-x}$, $x>0$, and $A_n$ is a monotone increasing sequence, we have 
\begin{align}\label{eqlogint}
e_{n,n+i}&=\prod_{k=n+1}^{n+i} (1-2p_k)\le e^{-(2p_{n+1}+\dots+2p_{n+i}) }
=e^{-2\left(\frac{A_{n+1}}{n+1}+\dots+\frac{A_{n+i}}{n+i}\right) } \le e^{-2A_{n+1}\left(\frac1{n+1}+\dots+\frac1{n+i}\right) }
\nonumber \\ &
\le e^{-2A_{n+1}\int_{1}^i\frac{{\rm d}x}{n+x} }
=e^{-2A_{n+1} \ln\frac{n+i}{n+1}}=\left( \frac{n+1}{n+i}
\right)^{2A_{n+1}},
\end{align}
since $\sum_{j=a}^b \frac 1{j}\ge \int_{a}^b \frac {{\rm d} x}{x}$ for all integers $a,b$ with  $b>a\ge 2$. So
\begin{align*}
a_n = & 1+\sum_{i=1}^\infty e_{n,n+i} 
\le 1+\sum_{i=1}^\infty \left( \frac{n+1}{n+i} \right)^{2A_{n+1}}
\le 1+\int_{0}^\infty \left( \frac{n+1}{n+x} \right)^{2A_{n+1}}{\rm d} x\\
= &1+ \frac{(n+1)^{2A_{n+1}}}{2A_{n+1}-1}\frac{1}{n^{2A_{n+1}-1}}=1+\frac{n}{2A_{n+1}-1}\left(1+\frac{1}{n}\right)^{2A_{n+1}}\\
=&1+\frac{n}{2A_{n+1}-1}\left(1+\frac{1}{n}\right)^{n\cdot 2p_{n+1}}\left(1+\frac{1}{n}\right)^{2p_{n+1}}
=1+\frac{n}{2A_{n+1}-1}e^{2p_{n+1}}(1+o(1))\\
=&1+\frac{n(1+O(p_{n+1}))(1+o(1))}{2A_{n+1}-1}
\end{align*}
for large $n$. Since $A_{n+1}\to\infty$, we have $a_n=o(n)$.\\

\medskip
\noindent\underline{Step (ii).} There exists an $N\ge 1$ such that for all $n\geq N$ we have $p_n\geq \frac{1}{n}$ and $q_n\ge \frac14$. Also, $a_n\geq 1$. Hence, for $m$ large enough, $v_m=\sum_{n=1}^m 4 a_n^2p_nq_n\geq \overset{m}{\underset{n=N}{\sum}}4p_n q_n\geq \sum_{n=N}^m\frac{1}{n}\to\infty$ as $m\rightarrow \infty$.\\

\medskip
\noindent\underline{Step (iii).} Since $p_n\downarrow 0$, one has
$$
p_n a_n=p_n\left[1+(1-2p_{n+1})+(1-2p_{n+1})(1-2p_{n+2})+\dots\right]
\ge p_n\sum_{k=0}^\infty (1-2p_n)^k=\frac 12.
$$
From its definition it follows that $v_n$ is monotone; we also know that $v_n\to\infty$.
Hence, by the Stolz--Ces\`aro Theorem\footnote{This is the discrete version of L'Hospital's rule --- see e.g.\ Problem~70 in~\cite{PSz98}.}, we have
\begin{align}\label{DiscrLop}
\limsup_{n\to\infty} \frac{a_n^2}{v_n}
&\le \limsup_{n\to\infty} \frac{a_{n}^2-a_{n-1}^2}{{v_{n}-v_{n-1}}}
= \limsup_{n\to\infty} \frac{(a_{n}+a_{n-1})(a_{n}-a_{n-1})}{4p_nq_n a_n^2}
\\
&\le \limsup_{n\to\infty} \frac{(a_{n}+a_{n-1})(a_{n}-a_{n-1})}{2a_n}
\le \frac{1}{2}\limsup_{n\to\infty} (a_{n}-a_{n-1}),\nonumber
\end{align}
since $4p_nq_n a_n^2= (2p_n a_n) \cdot q_n \cdot 2a_n$, and $q_n\to 1$, $p_n a_n\ge 1/2$,  $a_{n-1}\le a_n$.
Next,
\begin{equation}\label{next}
\begin{aligned}
a_n-a_{n-1}&=
\sum_{i=1}^\infty \left[e_{n,n+i}-e_{n-1,n-1+i}\right]
=
\sum_{i=1}^\infty \left[e_{n,n+i-1}(1-2p_{n+i})-(1-2p_n)e_{n,n-1+i}\right]\\
&=2\sum_{i=1}^\infty (p_n-p_{n+i})e_{n,n+i-1}.
\end{aligned}
\end{equation}
We have (e.g. by integrating by parts)
\begin{align*}
\sum_{i=1}^\infty \frac{i}{(n-1+i)^{2A_{n+1}+1}}
\le & 
\int_0^\infty \frac{x\,{\rm d}x}{(n-1+x)^{2A_{n+1}+1}}=
 \frac{1}{2A_{n+1}(2A_{n+1}-1) (n-1)^{2A_{n+1}-1}}.
\end{align*}
From the monotonicity of $p_n$ and $np_n$, we get $p_n\ge p_{n+i}$ and $\frac{p_{n+i}}{p_n}\ge \frac{n}{n+i}.$ Then, from~\eqref{eqlogint} and~\eqref{next}, it follows that\footnote{The last equality is elementary: $\left(1+\frac 1{n-1}\right)^{2A_{n+1}}=(1+\frac{1}{n-1})^{(n-1)2p_{n+1}} \cdot (1+\frac{1}{n-1})^{4p_{n+1}}= O(e^{2p_{n+1}})=O(1+2p_{n+1})=O(1+p_n)$. }
\begin{align*}
0&\le \frac{a_n-a_{n-1}}{2p_n}= \sum_{i=1}^\infty \left(1-\frac{p_{n+i}}{p_n}\right)e_{n,n+i-1}
\le \sum_{i=1}^\infty \frac{i}{n+i}\cdot 
\left(\frac{n+1}{n+i-1}\right)^{2A_{n+1}}\\
&\le \sum_{i=1}^\infty \frac{(n+1)^{2A_{n+1}}\cdot i}{(n-1+i)^{2A_{n+1}+1}} 
\le 
\frac{(n+1)^{2A_{n+1}}}{2A_{n+1}(2A_{n+1}-1) (n-1)^{2A_{n+1}-1}}\\
&=
\frac{(n-1)\left(1+\frac 1{n-1}\right)^{2A_{n+1}}}{2A_{n+1}(2A_{n+1}-1) }=\frac{(n-1)(1+O(p_n))}{4A_{n+1}^2 (1+o(1))}.
\end{align*}
Hence
\begin{align*}
0\le a_n-a_{n-1}\le 2p_n\frac{n+o(n)}{4A_{n+1}^2}
=\frac{A_n(1+o(1))}{2A_{n+1}^2}\le 
\frac{1+o(1)}{2A_{n+1}}\to 0,
\end{align*}
so the righthand side of~\eqref{DiscrLop} tends to zero.\\

\medskip
\underline{Step (iv).}
We show how, in our case, (iii)  implies (iv). Since $a_i$ increases in $i$, and $|\xi_i|\le 2$ gives $a_i^2\xi_i^2\leq 4a_i^2$,  we have 
$$
\{i: a_i^2\xi_i^2 \geq \varepsilon n \}\subset \{i: 4 a_i^2\geq \varepsilon n \}=\left\{i: i\geq (f^2(n))^{-1}
\left ( \varepsilon n/4\right)\right\},
$$
where $f(\cdot)$ is the linear interpolation such that $f(i)=a_i$, and here $v$ can be treated also as a positive strictly increasing function  on $[0,\infty)$ with $v(m)=v_m$, so both $(f^2)^{-1}, v^{-1}$ are well-defined, positive and strictly increasing. Using that $Z(n)=v^{-1}(n)$,  Drogin's condition \eqref{drogin(a)} will be verified if we show that
\begin{equation}\label{cooling.iv.7}
 v^{-1}(n)< (f^2(n))^{-1} \left(\varepsilon n/4\right),
\end{equation}
for $n$ large enough,  because then,  for $n$ large enough, $a_i^2\xi_i^2<\varepsilon n$ for  $i\leq Z(n)$, that is, $$\mathbf{1}_{\{a_i^2\xi_i^2>n\varepsilon\}}=0,\  1\leq i\leq Z(n).$$ 

Since  $a^2_m=o(v_m)$, i.e. $f^2(x)=o(v(x))$,  for this $\varepsilon$, there is an $M$ such that for  $l\geq M$, $\ds{f^2(l)}/{v(l)}<{\varepsilon}/4$, and for such an $M$, there is an $N$ such that for  $x\geq N$ we have $ v^{-1}(x)\geq M$. Hence, 
$$
\frac{f^2(v^{-1}(x))}{v(v^{-1}(x))}=\frac{f^2(v^{-1}(x))}{x}<\frac{\varepsilon}4, \quad \forall x\geq N,
$$
that is, \eqref{cooling.iv.7} holds for $n\ge N$. This completes the proof of (iv) and that of the theorem altogether.

\subsection{Proof of Theorem \ref{invar_nheatcool}}
We again use the martingale approximation approach of section \ref{mgale.appr}. Notice that 
\begin{equation}\label{a_n}
a_n=1+\overset{\infty}{\underset{i=0}{\sum}}\overset{n+i}{\underset{k=n+1}{\prod}}(1-2p_k).
\end{equation}
Without the loss of generality, we may assume that $0<a<p_n<b<1$. Then $r:=\max\{|2a-1|, |2b-1|\}<1,$ and
$$
\left|\overset{n+i}{\underset{k=n+1}{\prod}}(1-2p_k)\right|\leq r^i,
$$
which is why the sum in \eqref{a_n} is well-defined, that is, the $a_n$ are well-defined, for all $n\ge 1$. Furthermore,
\begin{equation*}
\begin{aligned}
&1+\overset{\infty}{\underset{i=0}{\sum}}\overset{n+i}{\underset{k=n+1}{\prod}}(1-2p_k)
\leq &1+\overset{\infty}{\underset{i=1}{\sum}}\overset{n+i}{\underset{k=n+1}{\prod}}|1-2p_k|
\leq &1+\overset{\infty}{\underset{i=1}{\sum}}r^i=\frac{1}{1-r},
\end{aligned}
\end{equation*}
which gives $|a_n|\leq \frac{1}{1-r}$ for all $n$.

Next, we prove that $v(m)\overset{m\rightarrow\infty}{\longrightarrow}\infty$, or equivalently, that $\sigma_n\overset{n\rightarrow\infty}{\longrightarrow} \infty$:\\
(i) If $p_n\leq 1/2$, $\forall n$, then $a_n>1$, $\forall n$, and we immediately have $v(m)\overset{m\rightarrow\infty}{\longrightarrow}\infty$.\\
(ii) Otherwise we have a subsequence $\{p_{n_k}\}_{n_k}$ such that $n_{k+1}-n_k>1$ and $p_{n_k}>1/2$, for all $n_k$. Notice that, by \eqref{a_n} and a direct computation, we have
$$(a_{n-1}-1)=(a_{n}-1)(1-2p_n),$$
and thus for the subsequence one has
$$(a_{n_k-1}-1)=(a_{n_k}-1)(1-2p_n).$$
So the two subsequences $\{a_{n_k-1}-1\}_{k\ge 1}, \{a_{n_k}-1\}_{k\ge 1}$ have opposite signs, hence we have a subsequence of $\{a_n\}_{n\ge 1}$ such that its terms are larger than $1$.  Consequently,  $v(m)\overset{m\rightarrow\infty}{\longrightarrow}\infty$.\\
Moreover, the condition that $\ds \lim_{n\to\infty} \frac 1n \sum_{i=1}^{Z(n)} a_i^2 \xi_i^2 \mathbf{1}_{\{a_i^2 \xi_i^2 >n\epsilon\}}=0$ is easy to verify, since our~$a_n$ are bounded.\\
In conclusion, the answers to (M) and to (INV.M) are both in the affirmative, yielding the invariance principle \eqref{Wiener.limit}.

\subsection{Proof of Theorem \ref{WLLN.by.comparison}}
Fix $a>0$ and let $N=N(a)$ be such that $a/N\le 1/2$ and that also $a/n<p_n$ holds for all $n>N$. 
Define $\hat p_n$ so that it coincides with~$p_n$ for $n\le N$ and $\hat p_n=a/n$ for $n>N$.
Let $\hat S$ denote the walk for the sequence $(\hat p_n)$, and note that this walk depends on the parameter $a>0$. By the monotonicity established in the proof of Theorem \ref{Lower.Main.Thm},
$$
\Var\left(\frac{S_{n}}{n}\right)\le \Var\left(\frac{\hat S_{n}}{n}\right),\ n\ge 1.
$$
In~\cite{EV2018} it was shown that 
$$
\lim_{n\to\infty}\Var\left(\frac{\hat S_{n}}{n}\right)=\frac{1}{2a+1}
\quad\Longrightarrow\quad 
\limsup_{n\to\infty}\Var\left(\frac{S_{n}}{n}\right)\le\frac{1}{2a+1}.$$
Since $a>0$ was arbitrary, 
$$
\lim_{n\to\infty}\Var\left(\frac{S_{n}}{n}\right)=0,
$$ 
implying WLLN. $\qed$

\subsection{Proof of Theorem \ref{recmix}}
We first need a lemma.
\begin{lemma}\label{49.Lemma}
For every $m$, $n$ and $\ell$ such that $\ell>n\ge m\ge 1$ we have that \linebreak $\P(S_{\ell}\le S_n\mid Y_m)\ge \frac12(1-|e_{m,n+1}|)$.
\end{lemma}
\begin{proof}[Proof of Lemma] We do the proof for $Y_m=1$, for $Y_m=-1$ the proof is essentially the same. Writing out $e_{m,n+1}=E(Y_{n+1}\mid Y_m=1),$ one obtains
\begin{equation}\label{49}
\P(Y_{n+1}=1\mid Y_m=1)=\frac{1+e_{m,n+1}}{2};\hspace{3mm} \P(Y_{n+1}=-1\mid Y_m=1)=\frac{1-e_{m,n+1}}{2}.
\end{equation}
Next, we claim that
\begin{equation}\label{Stas.doesnot.like}
\frac12 \,\P(S_{\ell}\le S_n\mid Y_{n+1}=-1)+\frac12 \,\P(S_{\ell}\le S_n\mid Y_{n+1}=+1)\ge \frac 12.
\end{equation}
Indeed, let us start our walk at time $n$ instead of time zero at the location $S_n$, such that its first step is random and equals $1$ or $-1$ with equal probabilities. Then the LHS of~\eqref{Stas.doesnot.like} is the probability that $n-\ell$ times later this walk ends up at a position which is not larger than its initial position. By symmetry, this value is at least $1/2$.
By \eqref{49} and \eqref{Stas.doesnot.like} and Markov property,
\begin{align*}
&\P(S_{\ell}\le S_n\mid Y_m=1)=\sum_{j=\pm 1}\P(S_{\ell}\le S_n, Y_{n+1}=j\mid Y_m=1)\\
&\ \ =
 \sum_{j=\pm 1}\P(S_{\ell}\le S_n\mid Y_{n+1}=j)\,\P(Y_{n+1}=j|Y_m=1)\\
&\ \ \ge \min_{j=\pm 1} \P(Y_{n+1}=j|Y_m=1) \sum_{j=\pm 1} \P(S_{\ell}\le S_n\mid Y_{n+1}=j)\geq \frac{1-|e_{m,n+1}|}{2},
\end{align*}
as claimed.
\end{proof}

We now turn to the proof of Theorem \ref{recmix} and show e.g.\ that $\P(S_n<0\ \mathrm{i.o.}\mid \mathcal{F}_1)=1$; one can similarly show that $\P(S_n>0\ \mathrm{i.o.}\mid \mathcal{F}_1)=1$. It turns out that is enough to construct a sequence $(\ell_k)_{k\ge 0}$ such that $\P(S_{\ell_{i+1}}<0\mid \mathcal{F}_{\ell_{i}})\ge r$ holds with some $r>0$, and the statement then follows from the extended Borel-Cantelli Lemma. Below we define such a sequence recursively, for $r=1/6$. 

Let $\ell_0:=1$. Once $\{\ell_i,\ 0\le i\le k\}$ have been constructed, we construct $\ell_{k+1}$ as follows. By mixing, we can pick an  $N_{k}$ (depending on $\ell_k$ only) such that $|e_{\ell_{k},\ell}|<1/3$ for all $\ell\ge N_{k}$. By Lemma \ref{49.Lemma} then,  for  all $\ell\ge N_{k}$,
\begin{align}\label{eqq1}
\P(S_{\ell}<S_{\ell_k}\mid \mathcal{F}_{\ell_k})
\ge 1/3.
\end{align}
Using that $|S_{\ell_k}|\le \ell_k$ along with Assumption~\ref{spr}, 
$$
\limsup_{\ell\to\infty} \P(0\le S_{\ell}< S_{\ell_k}\mid \mathcal{F}_{\ell_k})\le \lim_{\ell\to\infty} \P(0\le S_{\ell}< \ell_k\mid \mathcal{F}_{\ell_k})=0,\ \mathrm{a.s.}
$$
Hence, $\exists\ \ell_{k+1}>\max\{\ell_k, N_k\}$ that depends only on $\ell_k$ such that
\begin{align}\label{eqq2}
\P(0\le S_{\ell_{k+1}}< S_{\ell_k}\mid\mathcal{F}_{\ell_k})\le 1/6.
\end{align}
By combining \eqref{eqq1} and \eqref{eqq2} we conclude that
$$
\P(S_{\ell_{k+1}}<0\mid\mathcal{F}_{\ell_k})\ge  1/3-1/6=1/6.
$$
The sought sequence $(\ell_k)_{k\ge 0}$ has thus been constructed.$\hfill\qed$

\subsection{Proof of Theorem \ref{rec.thm}}

Let $\tau_0:=0$ and  $$\tau_{n}:=\inf\{m> 2\tau_{n-1}:\ Y_m=-1\},\quad n=1,2,\dots.
$$
Since $\sum p_n=\infty$, by the Borel-Cantelli Lemma, there are infinitely many turns. As a result, with probability $1$,  all $\tau_n$ are  well-defined and finite. Clearly, $\tau_n\to\infty$, \ as $n\to\infty$.

Let 
$$
A_n:=\{Y_i=-1,\  \text{for all } i\in[\tau_n, 2\tau_{n}]\}\in \mathcal F_{2\tau_{n}}=:\mathcal G_n,
$$ 
and note that 
$
A_n\subseteq \{S_{2\tau_{n}}\le 0\}=:B_n.
$
If we show that $\sum_n \P(A_n\mid \mathcal G_{n-1})=\infty$ then by the extended Borel-Cantelli lemma (see Corollary 5.29 in~\cite{B1992}), it follows that $\P(A_n\text{ i.o.})=1$; hence $\P(B_n\text{ i.o.})=1$, and so $\P(S_n\le 0\text{ i.o.})=1$. 

Now, for $n\ge n_0$, 
$$
\P(A_n\mid \mathcal G_{n-1},\tau_n=k)=
\left(1-\frac c{k+1}\right)
\left(1-\frac c{k+2}\right)
\dots{}\left(1-\frac c{2k}\right),
$$ 
when $k$ is admissible (i.e. the condition has positive probability). Since obviously $\tau_n\ge n$, we know that this is never the case for $k<n$.

Since the product on the right hand side tends\footnote{A more detailed calculation shows that the RHS equals $2^{-c}\left[1-\frac{c(c-1)}{4k}+O(k^{-2})\right]$ but we do not need it here.} to $2^{-c}$, as $k\to\infty$, and  only $k$'s for which $k\ge n$ are admissible,
$$
\P(A_n\mid \mathcal G_{n-1},\tau_n=k)\ge 2^{-c-1}
$$ 
holds  for all large enough $n$ and admissible $k$'s.
Thus $
\P(A_n\mid \mathcal G_{n-1})\ge 2^{-c-1}
$
holds for all large enough $n$, and we are done.
  A completely symmetric argument shows that also $\P(S_n\ge 0\text{ i.o.})=1$, thus proving the recurrence of the walk $S$.

A similar proof, left to the reader, establishes that the scaling limit (zigzag process) is recurrent as well.

\section{Appendix} Here we invoke some background on random measures that we utilized in the proof of Proposition \ref{N_and_Phi}. Much more material on random measures can be found in  \cite{KO2017}.

Assume that we are given  a complete separable metric space $S$. 
\begin{definition}[Dissecting subsets]
Denote by $\widehat{S}$  the set of all bounded Borel sets of~$S$. A subset $\Ic \subset \widehat{S}$ is called \emph{dissecting} if
\begin{itemize}
\item[(a)] every open set $G\subset S$ is a countable union of sets in $\Ic$;
\item[(b)] every set $B\in \widehat{S}$ is covered by finitely many sets in $\Ic$.
\end{itemize}
\end{definition} 
The following lemma is a useful result concerning  the weak convergence of random measures. (The measures are equipped with the  vague topology, recall Notation 1.)

\begin{lemma}[Theorem 4.11 in \cite{KO2017}]\label{lemma_vague}
Let $\xi, (\xi_n)_n$ be random measures on $S$ and let $E$ denote the expectation for $\xi$.
Furthermore, let 
\begin{enumerate}
\item  $\widehat{C}_s$ be the set of all continuous compactly supported functions on $S$; 
\item  $\widehat{S}_{E\xi}$  be the class of all bounded sets $A \subset S$ with $E\xi(\partial A)=0$;
\item $\widehat{\Ic}_+$  be the set of all non-negative simple $\Ic-measurable$ functions for a fix dissecting semi-ring $\Ic\subset \widehat{S}_{E\xi}$.
\end{enumerate}
 Then, as $n\to\infty$, 
$\xi_n  \overset{vd}{\longrightarrow} \xi$ if and only if $\xi_n(f)\overset{d}{\longrightarrow} \xi(f)$ holds either for all $ f\in \widehat{C}_s$ or for all $f\in\widehat{\Ic}_+$.
\end{lemma}
\vspace{0mm}
\noindent{\bf Acknowledgements:} We are grateful to  M\'arton Bal\'azs,  Bertrand Cloez, Edward Crane, Laurent Miclo and Sean O'Rourke for helpful discussions, and also to an anonymous referee for his/her close reading of the manuscript. S.V. thanks the University of Colorado for the hospitality during his visit in May 2018.


\begin{thebibliography}{EVW19}
\bibitem{BBC2017} Bena\"{\i}m, M.; Bouguet, F.; Cloez, B.
Ergodicity of inhomogeneous Markov chains through asymptotic pseudotrajectories. 
Ann. Appl. Probab. 27 (2017), no. 5, 3004--3049.
\bibitem{BC2018} Bouguet, F.; Cloez, B. Fluctuations of the empirical measure of freezing Markov chains. Electron. J. Probab. 23 (2018), Paper No. 2.
\bibitem{B1992} Breiman, L. Probability. Corrected reprint of the 1968 original. Classics in Applied Mathematics, 7. Society for Industrial and Applied Mathematics (SIAM), Philadelphia, PA, 1992.
\bibitem{Da1970} Davydov, Ju. A. The invariance principle for stationary processes. (Russian) Teor. Verojatnost. i Primenen. 15 1970 498--509. 
\bibitem{Da1973} Davydov, Ju. A. Mixing conditions for Markov chains. (Russian) Teor. Verojatnost. i Primenen. 18 (1973), 321--338. 
\bibitem{D1972} Drogin, R. An invariance principle for martingales. Ann. Math. Statist. 43 (1972), 602--620.
\bibitem{EV2018} Engl\"ander, J.; Volkov, S. Turning a coin over instead of tossing it. J. Theor. Probab. 31 (2018), 1097--1118.
\bibitem{KO2017} Kallenberg, O. Random measures, theory and applications.  Probability Theory and Stochastic Modeling, 77. Springer, Cham, 2017.
\bibitem{KS1991} Karatzas, I.; Shreve, S. E.
Brownian motion and stochastic calculus. 
Second edition. Graduate Texts in Mathematics, 113. Springer-Verlag, New York, 1991.
\bibitem{PSz98} P\'olya, G.; Szeg\H{o}, G. Problems and theorems in analysis. I. Series, integral calculus, theory of functions. Translated from the German by Dorothee Aeppli. Reprint of the 1978 English translation. Classics in Mathematics. Springer-Verlag, Berlin, 1998.
\end{thebibliography}
\end{document}